\documentclass[12pt]{amsart}
\usepackage{amsfonts}
\usepackage{amssymb}
\usepackage{amsmath}
\usepackage{tikz-cd}
\usepackage{enumerate}
\usepackage{relsize}
\usepackage{scalerel}
\usepackage{bm}
\usepackage{amsrefs}
\usepackage[makeroom]{cancel}
\usepackage{appendix}
\usepackage{comment}
\usepackage{centernot}
\usepackage{fullpage}
\usepackage{hyperref}

\newtheorem{theorem}{Theorem}[section]
\newtheorem{lemma}[theorem]{Lemma}
\newtheorem{proposition}[theorem]{Proposition}
\newtheorem{corollary}[theorem]{Corollary}

\newtheorem*{theorem*}{Theorem}

\theoremstyle{definition}
\newtheorem{definition}[theorem]{Definition}

\newtheorem*{definition*}{Definition}

\theoremstyle{remark}

\newtheorem{example}[theorem]{Example}

\usepackage[T3,T1]{fontenc}
\DeclareSymbolFont{tipa}{T3}{cmr}{m}{n}
\DeclareMathAccent{\bend}{\mathalpha}{tipa}{16}

\renewcommand{\H}{\mathcal H}

\newcommand{\Q}{\mathcal Q}

\newcommand{\V}{\mathcal V}
\newcommand{\W}{\mathcal W}
\newcommand{\X}{\mathcal X}
\newcommand{\Y}{\mathcal Y}
\newcommand{\Z}{\mathcal Z}

\newcommand{\BB}{\mathbb B}
\newcommand{\CC}{\mathbb C}

\newcommand{\NN}{\mathbb N}

\newcommand{\RR}{\mathbb R}

\newcommand{\fF}{\mathfrak F}

\renewcommand{\>}{\rangle}
\renewcommand{\:}{\colon}
\newcommand{\To}{\rightarrow}

\newcommand{\subsetof}{\subseteq}

\newcommand{\tensor}{\otimes}
\newcommand{\iso}{\cong}

\newcommand{\union}{\cup}

\newcommand{\EV}{\quad\Longleftrightarrow\quad}
\newcommand{\IMP}{\quad\Longrightarrow\quad}

\newcommand{\ran}{\mathop{\mathrm{ran}}}

\renewcommand{\above}{\sqsupseteq}
\newcommand{\below}{\sqsubseteq}
\newcommand{\cat}{\mathbf}

\newcommand{\atomof}{\text{\raisebox{1.2pt}{\smaller\,$\propto$\,}}}

\newcommand{\Tr}{\mathrm{Tr}}
\newcommand{\At}{\mathrm{At}}

\newcommand{\Eval}{\mathrm{Eval}}

\newcommand{\Set}{\mathbf{Set}}
\newcommand{\POS}{\mathbf{POS}}

\newcommand{\qRel}{\mathbf{qRel}}
\newcommand{\qSet}{\mathbf{qSet}}

\newcommand{\qPOS}{\mathbf{qPOS}}

\newcommand{\FdHilb}{\mathbf{FdHilb}}

\newcommand{\Pow}{\mathrm{Pow}}

\newcommand{\counit}{\epsilon}

\newcommand{\Inc}{\mathrm{Inc}}
\newcommand{\qPow}{\mathrm{qPow}}

\allowdisplaybreaks

\title{A category of quantum posets}
\author{Andre Kornell}
\address{Department of Computer Science, Tulane University, New Orleans, Louisiana 70118}
\email{akornell@tulane.edu}
\author{Bert Lindenhovius}
\address{Institute of Mathematical Methods
in Medicine and Databased Modelling, Johannes \mbox{Kepler} University, Linz, Austria}
\email{albertus.lindenhovius@jku.at}
\author{Michael Mislove}
\address{Department of Computer Science, Tulane University, New Orleans, Louisiana 70118}
\email{mislove@tulane.edu}
\thanks{This research was supported by AFOSR under MURI grant FA9550-16-1-0082, entitled “Semantics, Formal Reasoning, and Tool Support for Quantum Programming”. The second author is currently supported by the bilateral Austrian Science Fund (FWF)
Project I 4579-N and Czech Science Foundation (GA\v{C}R) Project 20-09869L “The many facets of orthomodularity”.}

\begin{document}

\begin{abstract}
We investigate a category of quantum posets that generalizes the category of posets and monotone functions. Up to equivalence, its objects are hereditarily atomic von Neumann algebras equipped with quantum partial orders in Weaver's sense. We show that this category is complete, cocomplete and symmetric monoidal closed. As a consequence, any discrete quantum family of maps from a discrete quantum space to a partially ordered set is canonically equipped with a quantum preorder. In particular, the quantum power set of a quantum set is canonically a quantum poset. We show that each quantum poset embeds into its quantum power set in complete analogy with the classical case.
\end{abstract}

\maketitle

This paper defines a category of quantum posets, along the lines of the category of compact quantum spaces \cite{Woronowicz}. This falls under the rubric of `noncommutative’ mathematics, a program to investigate noncommutative operator algebras from the premise that, like commutative operator algebras, they consist of complex-valued functions on `quantum' spaces of various kinds \cite{GraciaBondiaVarillyFigueroa}. Thus, a quantum space is intuitively the spectrum of a noncommutative operator algebra. These quantum spaces are not spaces in the literal sense: they are not sets of points equipped with additional structure. Rather, they are just a way of speaking about noncommutative operator algebras, and in that sense, one may say that they don't really exist. In this, they are comparable to locales \cite{Johnstone}.

Formally, quantum spaces are often defined to be the objects of the opposite of some category of operator algebras, echoing Gelfand duality. For example, quantum sets may be identified with the objects of the opposite of the category of hereditarily atomic von Neumann algebras and unital normal $*$-homomorphisms \cite{Kornell2}. The fact that we work in the quantum setting means that we work with generalizations of definitions that are familiar in the classical setting, but which can be justified only informally; a quantum partial order is not a partial order in the literal sense. Our basic building blocks are hereditarily atomic von Neumann algebras and Weaver's quantum relations \cite{Weaver1}. 

A \emph{quantum relation} on a von Neumann algebra $M\subseteq L(H)$ is an ultraweakly closed subspace $V$ of the space $L(H)$ of bounded operators on the Hilbert space $H$ such that $M'\cdot V\cdot M'\subseteq V$. Weaver distilled this notion of a quantum relation from his work with Kuperberg on quantum metrics \cite{KuperbergWeaver} and observed that a number of familiar conditions on binary relations such as reflexivity, symmetry, antisymmetry and transitivity have natural analogues for these quantum relations. Hence, quantum relations yield a method for finding quantum generalizations of several classes of mathematical structures. One can define a quantum generalization of functions in this way. The resulting category of von Neumann algebras and quantum functions turns out to be dual to the category of von Neumann algebras and normal $*$-homomorphisms \cite{Kornell-qf}. We give some examples of known noncommutative structures that can be described in the framework of quantum relations. 

Quantum graphs \cite{Weaver1}*{Definition 2.6(d)}\cite{DuanSeveriniWinter} have been the subject of recent research interest \cite{Stahlke, Weaver3, MustoReutterVerdon, BrannanGanesanHarris, ChirvasituWasilewski}. These quantum graphs are very similar to the quantum posets that we study here, and many of the techniques that we use should apply to a category of quantum graphs that is similarly defined. For example, we expect that the existence of a quantum graph homomorphism \cite{MancinskaRoberson} between two finite simple graphs should have a natural formulation in a category of quantum graphs via Lemma \ref{lem:monoidal closure lemma} or a variant. Quantum graph homomorphisms are the founding example in the study of winning nonsignaling strategies for synchronous games \cite{MancinskaRobersonVaritsiotis,  PaulsenSeveriniStahlkeTodorovWinter,  AtseriasMancinskaRobersonSamalSeveriniVarvitsiotis}.
However, the relevant graph homomorphisms respect an irreflexive binary relation, whereas  monotone maps -- the relevant quantum poset morphisms -- respect a reflexive binary relation. As a consequence, many of the arguments in the present paper do not directly apply to this category of quantum graphs.

Quantum metric spaces were introduced by Kuperberg and Weaver \cite{KuperbergWeaver} and form another example of a noncommutative structure that can be described in terms of quantum relations. In fact, the notion of a quantum relation \cite{Weaver1} was extracted by Weaver from this work with Kuperberg. They observed that on any quantum graph, one can define a quantum metric that is analogous to the standard graph metric. As one example of these quantum graph metrics they identified quantum Hamming distance, but they also observed that the operator systems used by Knill, Laflamme and Viola as an error model \cite{KnillLaflammeViola} yield quantum graph metrics. Further, they made more connections with existing noncommutative structures by showing that any spectral triple $(A,H,D)$ induces a quantum metric on $A''$, just as any Riemannian metric on a manifold induces a metric on the underlying topological space. 


Our main motivation for investigating the category of quantum partial orders lies in the problem of constructing denotational semantics for quantum programming languages. Here, a denotational semantics of a programming language is a translation of any phrase in the programming language to a mathematical function in such a way that this function is the composition of the functions that correspond to subphrases \cite{ScottStrachey}. Since debugging quantum programs directly is impractical, denotational semantics provides an alternative method for the verification of quantum programs. Complete partial orders, i.e., those partial ordered sets in which the supremum of any monotonically ascending sequence exists, form a category used in the denotational semantics of ordinary programming languages with recursion. Hence, it is plausible that quantizing complete partial orders yields the appropriate category for the semantics of quantum programming languages with recursion. This quantization of complete partial orders requires first establishing the categorical properties of quantum posets. It will be shown in subsequent work that this indeed leads to a satisfying denotational model of quantum programming languages with recursion; see \cite{KornellLindenhoviusMislove}.

Another motivation for the study of the categorical properties of quantum posets is the quantum power set, which is the main application in this contribution. Just like ordinary power sets are equipped with a natural order, the inclusion relation, we show that the quantum power set is equipped with a natural quantum order, which requires various categorical techniques. It is to be expected that the quantum power set and this natural order will play an important role in a possible definition of quantum topological spaces that are not necessarily locally compact Hausdorff, generalizing Kuperberg and Weaver's quantum metric spaces.

For intuition, we first consider a quantum relation on the von Neumann algebra $M_n(\CC)$, which is just a subspace $V$ of $ M_n(\CC)$. Intuitively, $V$ relates two vector states $x_1, x_2 \in \CC^n$ if $\<x_2 | v x_1\> \neq 0$ for some $v \in V$. This intuition underlies the first appearance of quantum relations in the special case of quantum graphs \cite{DuanSeveriniWinter}, where Duan, Severini and Winter draw an analogy between the confusibility graph of a classical channel and an operator system obtained from the Kraus operators of a quantum channel. 

A quantum relation $V$ on $M_n(\CC)$ is a \emph{quantum preorder} if $1 \in V$ and $v_1 v_2 \in V$ for all $v_1, v_2 \in V$, in other words, if $V$ is a unital subalgebra of $M_n(\CC)$ \cite{Weaver1}*{Definition 2.6(b)}. Physically, a quantum preorder may be understood as encoding the possibility of transitioning from one state of a quantum system to another in multiple steps, each taken from some set of quantum channels \cite{Weaver2}*{Section 1}. A system that is prepared in vector state $x_1 \in \CC^n$ may later be measured to be in vector state $x_2 \in \CC^n$ if and only if $\<x_2 | v x_1\> \neq 0$ for some $v \in V$. Because the number of individual transitions is arbitrary, the subspace $V$ is an algebra, and because that number maybe zero, the subspace $V$ should contain the identity.

Identifying states that are equivalent in this regard corresponds to viewing $V$ as a quantum relation on the von Neumann algebra $M:= (V\cap V^\dagger)'$, where $\vphantom M^\dagger$ denotes the conjugate transpose and $\vphantom M '$ denotes the commutant. By the double commutant theorem, $V \cap V^\dagger = M'$; this is a quantum generalization of antisymmetry because $M'$ is the identity quantum relation on $M$. This motivates the definition that a \emph{quantum partial order} $V$ on a von Neumann algebra $M\subsetof M_n(\CC)$ is a unital subalgebra $V \subsetof M_n(\CC)$ that satisfies $V \cap V^\dagger = M'$. Indeed, if $V$ is a quantum partial order on $M$, then $M = (V \cap V^\dagger)'$.

If $M$ is commutative, then the quantum partial orders on $M$ are in one-to-one correspondence with partial orders on the set of pure states of $M$. In the noncommutative case, there is no such one-to-one correspondence. A quantum partial order $V$ on $M$ should be understood intuitively as being a fictitious order on the fictitious Gelfand spectrum of $M$, a finite quantum space in the sense of Woronowicz \cite{Woronowicz}.

The given physical intuition for quantum posets also explains their application to the semantics of recursion in quantum computation. Viewing each quantum partial order as encoding the possible transitions of a nearly static physical system, we may interpret it as an information order with higher states carrying less information about the initial state of the system than lower ones. Information orders are used to model recursion in classical computation \cite{LarsenWinskel, Scott}. Quantum partial orders may similarly be used to model recursion in quantum computation, and our results on the category of quantum posets serve to underpin this application \cite{KornellLindenhoviusMislove}.

A von Neumann algebra equipped with a quantum partial order could be called a quantum poset, in which case the von Neumann algebra plays the role of a generalized set. Of course, von Neumann algebras are typically understood as quantum generalizations of measure spaces, not sets. However, hereditarily atomic von Neumann algebras, i.e., those of the form $ \bigoplus_{i \in I} M_{d_i}(\CC)$, may be viewed as a quantum generalization of discrete topological spaces \cite{PodlesWoronowicz} or equivalently, sets \cite{Kornell2}. Thus, in this paper, we investigate hereditarily atomic von Neumann algebras that are equipped with a quantum partial order.

More precisely, we investigate an equivalent class of objects. Inspired by the very simple form of a hereditarily atomic von Neumann algebra, we may define a quantum set $\X$ to be simply a set of finite-dimensional Hilbert spaces, called its atoms \cite{Kornell2}*{Definition 2.1}. To each quantum set $\X$, we associate the hereditarily atomic von Neumann algebra $\ell^\infty(\X):= \bigoplus\{L(X) : X\text{ is an atom of }\X\}$; this is the equivalence between the two notions. Intuitively, the atoms of $\X$ are not its elements but rather its indecomposable subsets. One imagines that a quantum set consisting of a single $d$-dimensional atom in fact consists of $d^2$ elements that are inextricably clumped together. For this reason, we use the notation $X\atomof\X$ to express that $X$ is an atom of $\X$.

In the special case of hereditarily atomic von Neumann algebras, Weaver's quantum relations also have a convenient, explicit characterization \cite[A.2]{Kornell3}. A binary relation $R$ from a quantum set $\X$ to a quantum set $\Y$ may be defined to be simply a choice of subspaces $R(X,Y) \subsetof L(X,Y)$, for atoms $X$ of $\X$ and $Y$ of $\Y$ \cite{Kornell2}*{Definition 3.1}. Together, quantum sets and the binary relations between them form a dagger-compact category \cite{Kornell2}*{Theorem 3.6}. A dagger-compact category is a rigid symmetric monoidal category equipped with an involution that is compatible with its monoidal structure \cite{AbramskyCoecke}*{Definition 12}. This dagger-compact category, which we denote by $\cat{qRel}$, forms the basic setting of this paper.

The structure of the category $\qRel$ may be conveniently defined by intuiting each quantum relation as a matrix of operator subspaces. Indeed, if $\X$ and $\Y$ have finitely many atoms, $X_1, \ldots, X_n$ and $Y_1, \ldots, Y_m$, respectively, then a binary relation $R$ from $\X$ to $\Y$ can be pictured as a matrix whose columns correspond to the atoms of $\X$ and whose rows correspond to the atoms of $\Y$:
\begin{equation*}
\begin{aligned}
R =
\left[
\begin{matrix}
R(X_1, Y_1) & R(X_2, Y_1) & \cdots & R(X_n, Y_1) \\
R(X_1, Y_2) & R(X_2, Y_2) & \cdots & R(X_n, Y_2) \\
\vdots & \vdots & \ddots & \vdots \\
R(X_1, Y_m) & R(X_2, Y_n) & \cdots & R(X_n, Y_m) 
\end{matrix}
\right]
\end{aligned}\;.
\end{equation*}
If $\X$ and $\Y$ have infinitely many atoms, then a binary relation $R$ from $\X$ to $\Y$ can be intuited similarly as an infinite matrix. 

Basic operations on binary relations between quantum sets are then loosely analogous to basic operations on matrices. The composition of a binary relation $R$ from $\X$ to $\Y$ and a binary relation $S$ from $\Y$ to $\Z$ is defined by $(S \circ R)(X, Z) = \bigvee_{Y \atomof \Y} S(Y,Z) \cdot R(X, Y)$, where $\cdot$ denotes the obvious product of operator subspaces and $\bigvee$ denotes their obvious sum. Similarly, the adjoint of a binary relation $R$ from $\X$ to $\Y$ is defined by $R^\dagger(Y,X) = R(X,Y)^\dagger$, which recalls the definition of the Hermitian adjoint $\dagger$. A number of basic operations can be defined entry by entry, including $(\neg R)(X, Y) = R(X,Y)^\perp$, $(R_1 \wedge R_2)(X, Y) = R_1(X, Y) \cap R_2(X,Y)$, and $(R_1 \vee R_2)(X,Y) = R_1(X,Y) \vee R_2(X,Y)$. This choice of symbols reflects the fact that binary relations from $\X$ to $\Y$ form an orthomodular lattice with $R_1 \leq R_2$ if and only if $R_1(X,Y) \subseteq R_2(X,Y)$ for all atoms $X \atomof \X$ and $Y \atomof \Y$.

We now define quantum posets and monotone functions in terms of this structure, and state the main result of the paper:

\begin{definition*}
A quantum poset is a quantum set $\X$ equipped with a binary relation $R$ on $\X$ that satisfies $R \circ R \leq R$ and $R \wedge R^\dagger = I_\X$, where $I_\X$ is the identity binary relation on $\X$, i.e., the relation defined by $I_\X(X,X)=\CC 1_X$ for each atom $X$ of $\X$, and vanishing elsewhere.
\end{definition*}
\noindent Here, $R \circ R \leq R$ expresses the fact that $R$ is transitive, while $R \wedge R^\dagger = I_\X$ indicates that $R$ is reflexive. 

\begin{definition*}
A monotone function from a quantum poset $(\X,R)$ to a quantum poset $(\Y,S)$ is a binary relation $F$ from $\X$ to $\Y$ such that $F^\dagger \circ F \geq I_\X$, $F \circ F^\dagger \leq I_\Y$ and $F \circ R \leq S \circ F$.
\end{definition*}
\noindent In the above definition, the first inequality expresses that the relation $F$ is defined everywhere, and the second expresses that it is single-valued. Finally, the last inequality expresses the monotonicity of $F$. 

\begin{theorem*}
The category $\cat{qPOS}$ of quantum posets and monotone functions is complete, cocomplete and symmetric monoidal closed. The full subcategory of those quantum posets that have only one-dimensional atoms is equivalent to the category of posets and monotone functions.
\end{theorem*}

One consequence of this theorem is that the quantum power set of a quantum set is canonically a quantum poset, as we now explain. To begin, quantum function spaces were first introduced by So\l tan \cite{Soltan} following Wang's definition of permutation compact quantum groups \cite{Wang}. Quantum function spaces were then constructed for arbitrary pairs of von Neumann algebras \cite{Kornell}, which lead to the definition of quantum function sets \cite{Kornell2}*{Definition 9.2}. Up to natural isomorphism, the quantum power sets that we work with are quantum functions sets in this sense.

To motivate the definition of the quantum power set, we first recall that a function from a quantum set $\X$ to a quantum set $\Y$ is just a binary relation $F$ such that $F^\dagger \circ F \geq I_\X$ and $F \circ F^\dagger \leq I_\Y$. Such functions are in one-to-one correspondence with unital normal $*$-homomorphisms $\ell^\infty(\Y) \to \ell^\infty(\X)$ \cite{Kornell2}. This class of binary relations defines the category $\cat{qSet}$ of quantum sets and functions, a subcategory of $\cat{qRel}$.

\begin{theorem*}
The inclusion functor $\Inc\: \cat{qSet} \to \cat{qRel}$ has a right adjoint $\qPow\: \cat{qRel} \to \cat{qSet}$. For each quantum set $\X$, we have that $\qPow(\X) \iso `\{0,1\}^{\X^*}$.
\end{theorem*}

\noindent The quantum set $`\{0,1\}$ is a quantum set that consists of two one-dimensional atoms, and the atoms of the dual quantum set $\X^*$ are exactly the duals of the atoms of $\X$. The expression $`\{0,1\}^{\X^*}$ denotes the internal hom from $\X^*$ to $`\{0,1\}$ in the category $\cat{qSet}$, which is symmetric monoidal closed. This theorem directly generalizes a familiar characterization of the power set functor in the classical setting \cite{JencovaJenca}; the distinction between a quantum set and its dual is a purely quantum phenomenon. Thus, we are led to the following definition:

\begin{definition*}
Let $\X$ be a quantum set. The quantum power set of $\X$ is $`\{0,1\}^{\X^*}$.
\end{definition*}

\noindent Using \cite{Kornell2}*{Theorem 7.4}, it is straightforward to show that the quantum powerset of $\X$ has a $d$-dimensional atom for every projection operator $p \in M_d(\ell^\infty(\X))$ that is irreducible in the sense that $p$ does not commute with a projection operator $q \in M_d(\CC)$ unless $q = 0$ or $q=1$. Here, we identify projection operators $p_1$ and $p_2$ if $p_1 = u^\dag p_2 u$ for some unitary $u$ in $M_d(\CC)$.

A different quantum power set construction occurs in Takeuti's quantum set theory \cite{Takeuti, Ozawa}. This quantum generalization of sets is not closely related to the one considered here. Roughly, within the framework of noncommutative geometry, Takeuti's quantum set theory refers to Boolean valued models over complete orthomodular lattices rather than complete Boolean algebras \cite{Jech}.

The canonical order on the quantum power set $`\{0,1\}^{\X^*}$ may be obtained by forming the inner hom from $\X^*$ to $`\{0,1\}$ in the category $\cat{qPOS}$ rather than in $\cat{qSet}$. We order the $\X^*$ trivially, and we order $\{0,1\}$ by $0 \sqsubset 1$. We use $\sqsubset$ and $\sqsubseteq$ to order the elements in a poset, reserving $<$ and $\leq$ for the canonical order on the binary relations between two quantum sets.

Classically, every poset may be embedded into its power set by mapping each element to its down set, and we establish a quantum analogue of this fact:

\begin{definition*}
An order embedding of a quantum poset $(\X,R)$ into a quantum poset $(\Y, S)$ is a function $F$ from $\X$ to $\Y$ such that $R = F^\dagger \circ S \circ F$. 
\end{definition*}

\begin{theorem*}
Let $(\X, R)$ be a quantum poset. Then, there exists an order embedding of $(\X,R)$ into $(`\{0,1\}^{\X^*},S)$, where $(`\{0,1\}^{\X^*},S)$ is the internal hom from $(`\{0,1\},`{\below})$ to $(\X^*, I_{\X^*})$.
\end{theorem*}

\noindent Less abstractly, an order embedding of $(\X,R)$ into $(`\{0,1\}^{\X^*},S)$ is essentially just an isomorphism between $(\X,R)$ and a subset of $(`\{0,1\}^{\X^*},S)$, in the obvious sense. This follows directly from \cite{Kornell2}*{Proposition 10.1} because any order embedding $F$ is injective by Lemma \ref{lem:order embeddings are injective and monotone}.


\section{Definitions and examples} We begin by defining orders and pre-orders on quantum sets, essentially Weaver's quantum pre-orders and quantum partial orders \cite{Weaver1}*{Definition 2.6}. We also define the notion of a monotone function between quantum posets, which reduces to the familiar notion in the classical case.

\begin{definition}\label{def:quantum preorder and order}
Let $\X$ be a quantum set. We call a binary relation $R$ on $\X$ a \emph{pre-order}, and $(\X,R)$ a \emph{quantum pre-ordered set} if it satisfies the following two axioms:
\begin{itemize}
    \item[(1)] $I_\X\leq R$ \; ($R$ is reflexive);
    \item[(2)] $R\circ R\leq R$ \; ($R$ is transitive).
\end{itemize}
If, in addition, $R$ satisfies
\begin{itemize}
    \item[(3)] $R\wedge R^\dag=I_\X$ \; ($R$ is antisymmetric),
\end{itemize}
then we call $R$ a \emph{order}, and $(\X,R)$ a \emph{quantum poset}. 
\end{definition}

\begin{example}\label{ex:trivial order}
Let $\X$ be a quantum set. Then, $I_\X$ is easily seen to be an order on $\X$, the \emph{trivial} order.
\end{example}

\begin{example}\label{ex:qPOS quote}
Let $(A,\sqsubseteq)$ be any poset. Then, $(`A,`{\sqsubseteq})$ is a quantum poset, because ordinary posets are defined in the category $\cat{Rel}$ of sets and binary relations by the same three axioms and the ``inclusion'' functor $`(-)\: \cat{Rel} \to \cat{qRel}$ preserves all the relevant structure \cite{Kornell2}*{Section III}. Furthermore, because this inclusion functor is full and faithful, every order on $`A$ is of this form.
\end{example}

\begin{example}\label{ex:non-classical quantum poset}
Let $A$ be a unital algebra of operators on a nonzero finite-dimensional Hilbert space $H$ that is anti-symmetric in the sense that $A \cap A^\dag = \CC 1$  \cite{Waclaw}\cite{Waclaw2}\cite{Weaver2}. Then, the equation $R(H,H) = A$ defines an order on $\H$, the quantum set whose only atom is $H$. Furthermore, every order on $\H$ is of this form, as a simple consequence of the definition of $\cat{qRel}$ and of all the relevant structure on this category.
\end{example}

\begin{example}
Let $\X$ be a quantum set with two atoms, $X_1$ and $X_2$. We may define an order $R$ on $\X$ as follows:
$$
\left(
\begin{matrix}
R(X_1, X_1)
&
R(X_2, X_1)
\\
R(X_1, X_2)
&
R(X_2, X_2)
\end{matrix}
\right)
=
\left(
\begin{matrix}
\CC1_{X_1}
&
0
\\
L(X_1, X_2)
&
\CC 1_{X_2}
\end{matrix}
\right).
$$
Intuitively, the structure of the quantum poset $(\X, R)$ may be described as follows: the atom $X_1$ represents a subset of $\X$ of cardinality $\dim(X_1)^2$ that is trivially ordered, the atom $X_2$ represents a subset of $\X$ of cardinality $\dim(X_2)^2$ that is trivially ordered, and every element of the first subset is below every element of the second.
\end{example}

We record the following basic facts about orders on quantum sets in a single lemma, whose proof we omit because it is straightforward.
\begin{lemma}\label{lem:basic facts on quantum orders}
Let $R$ be a pre-order on a quantum set $\X$. Then,
\begin{itemize}
\item[(1)] $R^\dag$ is a pre-order, and it is an order if $R$ is an order;
    \item[(2)] $R\circ R=R$;
\end{itemize}
\end{lemma}
The order $R^\dag$ is called the \emph{opposite} order, since for any ordinary poset $(S,\sqsubseteq)$, we have $`(\sqsubseteq)^\dag=`(\sqsupseteq)$ on $`S$.

\begin{lemma}\label{lem:orders generated antisymmetric subalgebras}
Let $R$ be a pre-order on $\X$ and let $X\atomof\X$. Then $R(X,X)$ is a unital subalgebra of $L(X):=L(X,X)$. Moreover, if $R$ is an order, then $R(X,X)$ is an antisymmetric subalgebra, i.e., $R(X,X)\wedge R(X,X)^\perp=\CC 1_X$.
\end{lemma}
\begin{proof}
Since $I_\X\leq R$, we have $\CC1_X=I_\X(X,X)\leq R(X,X)$, implying that $R(X,X)$ contains the identity element of $L(X)$. We also compute that \[ R(X,X)\cdot R(X,X)\leq\bigvee_{X'\atomof\X}R(X,X')\cdot R(X',X)=(R\circ R)(X,X)\leq R(X,X).\]
Therefore, $R(X,X)$ is an algebra. Finally, if $R$ is an order, then
\[\CC 1_X=I_\X(X,X)=(R\wedge R^\dag)(X,X)=R(X,X)\wedge R^\dag(X,X)=R(X,X)\wedge R(X,X)^\dag, \]
so $R(X,X)$ is antisymmetric.
\end{proof}

\begin{lemma}\label{lem:intersection of preorders is a preorder}
Let $\X$ be a quantum set, and let $\{R_{\alpha}\}_{\alpha\in A}$ be a collection of pre-orders on $\X$. Then, $R=\bigwedge_{\alpha\in A}R_\alpha$ is also a pre-order.
\end{lemma}
\begin{proof}
We have $I_\X\leq R_\alpha$ for each $\alpha\in A$, hence also $I_\X\leq\bigwedge_{\alpha\in A}R_\alpha=R$. 
Furthermore, using Lemma \ref{lem:intersection}, we find that
\[R\circ R=\left(\bigwedge_{\alpha\in A}R_\alpha\right)\circ\left(\bigwedge_{\beta\in A}R_\beta\right)\leq \bigwedge_{\alpha,\beta\in A}R_\alpha\circ R_\beta\leq \bigwedge_{\alpha\in A}R_\alpha\circ R_\alpha\leq \bigwedge_{\alpha\in A}R_\alpha=R. \qedhere\] 
\end{proof}

\begin{definition}\label{def:monotone map}
Let $(\X,R)$ and $(\Y,S)$ be quantum pre-ordered sets. Then, a \emph{monotone} function $F\:(\X,R)\to(\Y,S)$ is a function $F\:\X\to\Y$ that satisfies any of the following equivalent conditions: 
\begin{itemize}
    \item[(1)] $F\circ R\leq S\circ F$;
    \item[(2)] $F\circ R\circ F^\dag\leq S$;
    \item[(3)] $R\leq F^\dag\circ S\circ F$.
\end{itemize}
\end{definition}
The equivalences between these conditions follow directly from the definition of a function between quantum sets.

\begin{lemma}\label{lem:composition of monotone functions is monotone}
Let $(\X,R)$, $(\Y,S)$ and $(\Z,T)$ be quantum posets, and let $F:\X\to\Y$ and $G:\Y\to\Z$ be monotone. Then $G\circ F:\X\to \Z$ is monotone.
\end{lemma}
\begin{proof}
Since $F$ is monotone, we have $F\circ R\leq S\circ F$. Monotonicity of $G$ means that $G\circ S\leq T\circ G$. Hence $G\circ F\circ R\leq G\circ S\circ F\leq T\circ G\circ F,$
showing that indeed $G\circ F$ is monotone.
\end{proof}
Since the composition of two monotone functions is monotone, we obtain a category of quantum posets and monotone functions, which we notate $\qPOS$.

\begin{example}\label{ex:function with trivially ordered domain is monotone}
Let $\X$ be a trivially ordered quantum set (cf. Example \ref{ex:trivial order}) and let $(\Y,S)$ be a quantum poset. Then any function $F:\X\to\Y$ is monotone. Indeed, we have
\[ F\circ I_\X=I_\Y\circ F\leq S\circ F.\]
\end{example}

\begin{lemma}\label{lem:qPOS quote}
Let $(A,\sqsubseteq_A)$ and $(B,\sqsubseteq_B)$ be posets, and let $f:A\to B$ be a function. Then $f$ is monotone if and only if $`f$ is monotone.
\end{lemma}
\begin{proof}
The function $`f:`A\to`B$ is monotone if and only if $`f\circ `(\sqsubseteq_S) \leq `(\sqsubseteq_B) \circ `f$, or equivalently $`(f\circ (\sqsubseteq_A)) \leq `((\sqsubseteq_B)\circ f)$. The functor $`(-):\cat{Rel} \to\qRel$ preserves dagger monoidal structure \cite{Kornell2}*{Section 3}; hence $`f$ is monotone if and only if $f\circ(\sqsubseteq_A)\leq(\sqsubseteq_B)\circ f$. It remains only to show that this inequality is equivalent to the monotonicity of $f$.

Assume that $f$ is monotone, and let $(a,b)$ be a pair in the binary relation $f \circ (\below_A)$. It follows that $b = f(a')$ for some $a' \above_A a$. Since $f$ is monotone, we find that $b \above_B f(a)$. In other words, $(a,b)$ is in the binary relation $(\below_B) \circ f$. Therefore, $f \circ (\below_A) \leq (\below_B) \circ f$.

Now, assume that $f$ satisfies the inequality $f \circ (\below_A) \leq (\below_B) \circ f$, and let $a_1 \below_A a_2$. The pair $(a_1, f(a_2))$ is in $f \circ (\below_A)$ simply by definition of composition. By assumption, it is also in $(\below_B) \circ f$. Thus, $f(a_1) \below_A f(a_2)$. Therefore, $f$ is monotone.
\end{proof}

\begin{proposition}
The functor $`(-):\POS\to\qPOS$, given by $(A,\sqsubseteq)\mapsto (`A,`{\sqsubseteq})$ on objects and by $f\mapsto `f$ on morphisms is fully faithful.
\end{proposition}
\begin{proof}
The functor $`(-):\cat{Rel} \to\qRel$ is fully faithful, and it preserves dagger monoidal structure \cite{Kornell2}*{Section 3}. Because functions and, moreover, monotone functions are defined in terms of this dagger monoidal structure (Lemma \ref{lem:qPOS quote}), it follows that this functor restricts to a functor $\Set \to \qSet$ and, moreover, to a functor $\POS \to \qPOS$.
\end{proof}

\section{Subposets} The subsets of a quantum set $\X$ correspond to injections into $\X$ \cite{Kornell2}*{Proposition 10.1}. We show that the subsets of a quantum poset $\X$ similarly correspond to order embeddings.

\begin{lemma}\label{lem:pullback order}
Let $\X$ and $\Y$ be quantum sets, and let $S$ be a pre-order on $\Y$. Let $F:\X\to\Y$ be a function. Then $R=F^\dag\circ S\circ F$ is a pre-order on $\X$, and $F:(\X,R)\to (\Y,S)$ is monotone. Moreover, if $(\Y,S)$ is a quantum poset and $F$ is injective, then $(\X,R)$ is a quantum poset as well.
\end{lemma}

\begin{proof}
We have
\[ I_\X\leq F^\dag\circ F=F^\dag\circ I_\Y\circ F\leq F^\dag\circ S\circ F=R,\]
so $R$ satisfies the first axiom of a pre-order on a quantum set. Furthermore, we have 
\[R \circ R=F^\dag\circ S\circ F\circ F^\dag\circ S\circ F\leq F^\dag \circ S\circ S\circ F\leq F^\dag\circ S\circ F=R,\]
so $R$ also satisfies the second axiom and hence is a pre-order.

Next we show that $F$ is monotone: \[ F\circ R=F\circ F^\dag\circ S\circ F\leq I_\Y\circ S\circ F=S\circ F.\]

Now, assume that $F$ is injective. Furthermore, assume that $S$ satisfies the antisymmetry axiom, i.e., $S\wedge S^\dag=I_\Y$.
We check that $R$ satisfies the antisymmetry axiom too:
\[R\wedge R^\perp = (F^\dag\circ S\circ F)\wedge (F^\dag\circ S\circ F)^\dag=(F^\dag\circ S\circ F)\wedge (F^\dag\circ S^\dag\circ F)=F^\dag\circ(S\wedge S^\dag)\circ F=F^\dag\circ F=I_\X,\]   
where the second equality follows from Proposition \ref{prop:intersection}, the penultimate equality follows by the antisymmetry axiom for $S$, and the last equality follows from the injectivity of $F$.
\end{proof}

We will shortly apply Lemma \ref{lem:pullback order} to show that any pre-order on a quantum set $\Y$ restricts to pre-order on each subset $\X \subsetof \Y$. This relationship $\X \subsetof \Y$ is just that each atom of $\X$ is also an atom of $\Y$, and it comes with an inclusion function $J_\X \: \X \to \Y$. For $X \atomof \X$ and $Y \atomof \Y$, the operator subspace $J_\X(X,Y)$ is the span of the identity operator if indeed $X = Y$, and otherwise, it is zero.

\begin{definition}\label{def:induced order}
Let $(\Y,S)$ be a quantum poset. Then a \emph{subposet} of $\Y$ consists of a subset $\X\subseteq\Y$ equipped with order $R=J_\X^\dag\circ S\circ J_\X$, to which we refer as the \emph{induced order} on $\X$.
\end{definition}

The quantum generalization of the concept of a subposet leads to the quantum generalization of the notion of an order embedding:

\begin{definition}
Let $(\X,R)$ and $(\Y,S)$ be quantum pre-ordered sets. Then we call a function $J:\X\to \Y$ an \emph{order embedding} if $R=J^\dag\circ S\circ J$.
\end{definition}

We say that a function $F\: \X \to \Y$ is an \emph{injection} if $F^\dagger \circ F = I_\X$, a \emph{surjection} if $F \circ F^\dagger = I_\Y$, and a \emph{bijection} if both equalities hold \cite{Kornell2}*{Section 4}. When $\X \subsetof \Y$, the inclusion function $J_\X$ is an injection in this sense. Just as order embeddings between posets in the classical sense are monotone and injective, this happens to be true for the quantum case as well:

\begin{lemma}\label{lem:order embeddings are injective and monotone}
Let $(\X,R)$ and $(\Y,S)$ be quantum posets, and let $F:\X\to\Y$ be an order embedding. Then $F$ is both injective and monotone.
\end{lemma}
\begin{proof}
Monotonicity directly follows from Definition \ref{def:monotone map}. Since $R$ and $S$ are both orders, and $R=F^\dag\circ S\circ F$, we obtain
\[  I_\X=R\wedge R^\dag=  (F^\dag\circ S\circ F)\wedge (F^\dag\circ S^\dag\circ F)=F^\dag\circ (S\wedge S^\dag)\circ F=F^\dag\circ F,\]
where we used Proposition \ref{prop:intersection} in the penultimate equality.
\end{proof}

\begin{lemma}\label{lem:order embeddings compose to an order embedding}
Let $(\X,R)$, $(\Y,S)$, and $(\Z,T)$ be quantum posets, and let $F_1:\X\to\Y$ and $F_2:\Y\to\Z$ be order embeddings. Then the composition $F_2\circ F_1$ is an order embedding, too.
\end{lemma}
\begin{proof}
Since $F_1$ and $F_2$ are order embeddings, we have $R=F_1^\dag\circ S\circ F_1$ and $S=F_2^\dag\circ T\circ F_2$. Hence
$(F_2\circ F_1)^\dag\circ T\circ (F_2\circ F_1)=F_1^\dag\circ F_2^\dag\circ T\circ F_2\circ F_1=F_1^\dag\circ S\circ F_1=R$. \qedhere 
\end{proof}

\begin{definition}\label{def:order iso}
Let $(\X,R)$ and $(\Y,S)$ be quantum posets. A monotone map $F:\X\to\Y$ is called an \emph{order isomorphism} if it is bijective, and its inverse $F^\dag$ is monotone, too.
\end{definition}

\begin{proposition}\label{prop:order iso is surjective order embedding}
Let $(\X,R)$ and $(\Y,S)$ be quantum posets, and let $F:\X\to\Y$ be a function. Then the following statements are equivalent:
\begin{itemize}
    \item[(a)] $F$ is an order isomorphism;
    \item[(b)] $F$ is a surjective order embedding;
    \item[(c)] $F$ is a bijection such that $F\circ R=S\circ F$.
\end{itemize}  
\end{proposition}

\begin{proof}
First assume that $F$ is an order isomorphism. We have that
$R\leq F^\dag\circ S\circ F$ by the monotonicity of $F$.
Since $F^\dag$ is also monotone, we have
$ S\leq F\circ R\circ F^\dag.$
Since $F$ is injective, it follows that $F^\dag\circ S\circ F\leq F^\dag\circ F\circ R\circ F^\dag\circ F=R.$
Hence $R=F^\dag\circ S\circ F$, so $F$ is an order embedding, which is surjective since $F$ is a bijection. Therefore, (a) implies (b).

Assume that (b) holds. Then $R=F^\dag\circ S\circ F$. We have that $F\circ R=F\circ F^\dag\circ S\circ F=S\circ F$ by the surjectivity of $F$.
Moreover, since $F$ is an order embedding, it is injective by Lemma \ref{lem:order embeddings are injective and monotone}, and it is hence bijective. So (c) holds.

Finally, we show that (c) implies (a). Hence, let $F$ be a bijection such that $F\circ R=S\circ F$. This equality immediately gives that $F$ is monotone. Moreover, the bijectivity of $F$ yields $F^\dag\circ S=F^\dag\circ S\circ F\circ F^\dag=F^\dag\circ F\circ R\circ F^\dag=R\circ F^\dag,$
which implies that $F^\dag$ is monotone, too. We conclude that $F$ is an order isomorphism.
\end{proof}

\section{Monomorphisms and epimorphisms} We show that the monomorphisms of $\cat{qPOS}$ are exactly the injective monotone functions and that all extremal epimorphisms of $\cat{qPOS}$ are surjective. We will later use both results to show that $\cat{qPOS}$ is cocomplete. We do not characterize arbitrary epimorphisms.

\begin{lemma}\label{lem:monomorphisms in qPOS}
Let $(\X,R)$ and $(\Y,S)$ be quantum posets, and let $M:\X\to\Y$ be monotone. Then $M$ is a monomorphism in $\qPOS$ if and only if $M$ is injective.
\end{lemma}
\begin{proof}
Let $M:\X\to\Y$ be injective, let $(\W,T)$ be a quantum poset, and let $F,G:\W\to\X$ be two monotone functions such that $M\circ F=M\circ G$. Since $M$ is injective, it is a monomorphism in $\qSet$ (cf. \cite[Proposition 8.4]{Kornell2}); hence it follows that $F=G$. Thus $M$ is also a monomorphism in $\qPOS$.

We prove the converse by contraposition, so assume that $M$ is not injective. By \cite[Proposition 8.4]{Kornell2}, $M$ is not an monomorphism in $\qSet$. Hence, there is a quantum set $\W$ and there are functions $F,G:\W\to\X$ such that $F\neq G$, but $M\circ F=M\circ G$. Equip $\W$ with the trivial order $I_\W$. By Example \ref{ex:function with trivially ordered domain is monotone}, it follows that $F$ and $G$ are monotone. We conclude that $M$ is not a monomorphism in $\qPOS$.
\end{proof}

\begin{definition}\label{def:range of a function}
Let $F:\X\to\Y$ be a function from a quantum set $\X$ to a quantum set $\Y$. We define the \emph{range} of $F$ to be the subset
\[\ran F:=\Q\{Y\atomof\Y:F(X,Y)\neq 0\text{ for some }X\atomof\X\} \subsetof \Y.\]
We also define the binary relation $\overline F$ from $\X$ to $\ran F$ by $\overline F (X, Y) = F (X,Y)$, for $X \atomof \X$ and $Y \atomof \ran F$. It is routine to verify that $\overline F$ is a surjective function that satisfies $F = J_{\ran F} \circ \overline F$, where $J_{\ran F}\: \ran F \hookrightarrow \Y$ is the canonical inclusion \cite{Kornell2}*{Definition 8.2}.
\end{definition}

Thus, each function $F$ has a canonical factorization into an inclusion following a surjection. If $F$ is monotone, then both factors are also monotone, provided that we equip the range of $F$ with the induced order.

\begin{lemma}\label{lem:F-bar is monotone}
Let $(\X,R)$ and $(\Y,S)$ be quantum posets, and let $F$ be a monotone function from $(\X,R)$ to $(\Y,S)$. Then, $\overline F$ is a monotone function from $(\X, R)$ to $(\ran F, J^\dag \circ S \circ J)$, where $J = J_{\ran F}\: \ran F \hookrightarrow \Y$ is the canonical inclusion.
\end{lemma}

\begin{proof}
We reason that $\overline F \circ R = J^\dag \circ J \circ \overline F \circ R \leq J^\dag \circ F \circ R \leq J^\dag \circ S \circ F = J^\dag \circ S \circ J \circ \overline F$.
\end{proof}

It is easy to see that any surjective monotone function is an epimorphism in the category $\qPOS$. We do not show the converse; for our purposes, it is sufficient to show that any \emph{extremal} epimorphism is surjective. Recall that an epimorphism $E$ is said to be extremal if the only monomorphisms $M$ satisfying $E=M\circ F$ for some morphism $F$ are isomorphisms \cite{AdamekHerrlichStrecker}*{Definition 7.74}.

\begin{lemma}\label{lem:extremal epi in qPOS}
Let $(\X,R)$ and $(\Y,S)$ be quantum posets, and let $E:\X\to\Y$ be an extremal epimorphism in $\qPOS$. Then, $E$ is surjective.
\end{lemma}
\begin{proof}
The monotone function $E$ factors through $\ran E$ as $E = J_{\ran E} \circ \overline E$, with both factors being monotone for the induced order on $\ran E$. Because $E$ is an extremal epimorphism, the monomorphism $J_{\ran E}$ must be an isomorphsim in $\cat{qPOS}$, and therefore also an isomorphism in $\cat{qSet}$, i.e., a bijection. It clearly follows that $\ran E = \Y$. We conclude that $F = \overline F$ and, in particular, that $F$ is a surjection.
\end{proof}

\section{Order enrichment}\label{sec:quantum poset order enrichment}
We show that for all quantum posets $\X$ and $\Y$, the order on $\Y$ imposes an order on the hom set $\cat{qPOS}(\X, \Y)$, just as it does in the classical case. Let $(Y,\sqsubseteq)$ be a poset. Then, for any set $X$, we can order $\Set(X,Y)$ by $f\sqsubseteq g$ if and only if $f(x)\sqsubseteq g(x)$ for all $x\in X$. If we order the binary relations between two sets by inclusion, and regard $f$, $g$ and $\sqsubseteq$ as binary relations, this condition is equivalent to $g\leq (\sqsubseteq)\circ f$ and also to $(\sqsubseteq)\circ g\leq (\sqsubseteq) \circ f$. The next lemma shows that these last two inequalities between binary relations can be generalized to the quantum setting.

\begin{lemma}\label{lem:order between functions}
Let $\X$ be a quantum set, and let $(\Y,S)$ be a quantum poset. Then, we write $F\sqsubseteq G$ if $F,G\in\qSet(\X,\Y)$ satisfy any of the following equivalent conditions:
\begin{itemize}
    \item[(1)] $G\leq S\circ F$;
    \item[(2)] $S\circ G\leq S\circ F$;
    \item[(3)] $F\leq S^\dag\circ G$;
    \item[(4)] $G\circ F^\dag\leq S$.
\end{itemize}
Moreover, the relation $\sqsubseteq$ defines an order on $\qSet(\X,\Y)$.
\end{lemma}
\begin{proof}
The equivalence of conditions (1)-(4) follows easily from the definitions of a function and of an order. In order to show that $\sqsubseteq$ is an order on $\qSet(\X,\Y)$, it is straightforward to show that $\sqsubseteq$ is reflexive and transitive. For antisymmetry, let $F,G\in\qSet(\X,\Y)$, and assume that $G\sqsubseteq F$ and $F\sqsubseteq G$. Thus, $F\leq S\circ G$ and $G\leq S\circ F$, or equivalently, $F\leq S^\dag\circ G$. 
We now compute that $F\leq (S\circ G)\wedge (S^\dag\circ G)=(S\wedge S^\dag)\circ G=I\circ G=G,$
where we use Proposition \ref{prop:intersection} for the first equality, and axiom (3) of an order for the second equality. By Lemma \ref{lem:inequality between functions is equality}, we obtain $F=G$. 
\end{proof}

\begin{lemma}\label{lem:right multiplication is monotone}
Let $(\X,R)$ be a quantum poset, let $\Y$ and $\Z$ be quantum sets, and let $F:\Y\to\Z$ be a function. Let $K_1$ and $K_2$ be functions $\Z\to\X$ such that $K_1\sqsubseteq K_2$. Then, $K_1\circ F\sqsubseteq K_2\circ F$.
\end{lemma}
\begin{proof}
Let $K_1,K_2:\Z\to\X$ be such that $K_1\sqsubseteq K_2$. Then, $K_2\leq R\circ K_1,$ and
hence $
K_2\circ F\leq R\circ K_1\circ F,$
which expresses that $K_1\circ F\sqsubseteq K_2\circ F$.
\end{proof}

\begin{definition}\label{def:order on qPos}
Let $(\X,R)$ and $(\Y,S)$ be quantum posets. Then, we order $\qPOS(\X,\Y)$ by the induced order $\sqsubseteq$ from $\qSet(\X,\Y)$ (cf. Lemma \ref{lem:order between functions}). 
\end{definition}

We note that the order on $\qPOS(\X,\Y)$ only depends on $S$ and not on $R$. It is the same in the classical case, where the order on $\POS((X,\sqsubseteq_X),(Y,\sqsubseteq_Y))$ is defined by $f\leq g$ if and only if $f(x)\sqsubseteq_Y g(x)$ for all $x\in X$.

\begin{lemma}\label{lem:left multiplication by monotone map is monotone}
Let $(\X,R)$, $(\Y,S)$ and $(\Z,T)$ be quantum posets, and let $F:\Y\to\Z$ be monotone. Let  $K_1,K_2:\X\to\Y$ be functions (not necessarily monotone). If  $K_1\sqsubseteq K_2$, then $F\circ K_1\sqsubseteq F\circ K_2$
\end{lemma}
\begin{proof}
Recall that, by definition, $F$ is monotone if and only if $F\circ S\leq T\circ F$.
Let $K_1,K_2\in \qPOS( \X,\Y)$ such that $K_1\sqsubseteq K_2$. This means that $K_2\leq S\circ K_1$, so we compute that
$F\circ K_2\leq F\circ S\circ K_1\leq T\circ F\circ K_1$, and we
conclude that $F\circ K_1\sqsubseteq F\circ K_2$.
\end{proof}

\begin{lemma}\label{lem:homsets are monotone}
Let $(\X,R),(\Y,S)$ and $(\Z,T)$ be quantum posets, and let $F:(\Y,S)\to (\Z,T)$ be monotone. Then,
\begin{align*}
    \qPOS(\X,F): \qPOS( \X,\Y)  \to \qPOS( \X,\Z), \qquad K \mapsto F\circ K,
    \end{align*}
and 
\begin{align*}
\qPOS(F,\X):\qPOS( \Z,\X)  \to\qPOS(\Y,\X), \qquad K \mapsto K\circ F    
\end{align*}
   are monotone.
\end{lemma}

\begin{proof}
This follows from Lemmas \ref{lem:right multiplication is monotone} and \ref{lem:left multiplication by monotone map is monotone}.
\end{proof}

\begin{theorem}
The category $\qPOS$ is enriched over $\POS$.
\end{theorem}
\begin{proof}
For this statement to hold, we must have that each homset in $\qPOS$ is a poset, which is the case by Definition \ref{def:order on qPos}, and that composition is monotone in both arguments. The latter follows directly from Lemma \ref{lem:homsets are monotone}.
\end{proof}

\section{Completeness} We show that the category $\cat{qPOS}$ is complete by defining pre-orders on the limits that we have in $\cat{qSet}$. The main technical challenge is showing that these pre-orders are in fact partial orders.

\begin{definition}[cf. \cite{Weaver1}*{Definition 2.6(a)}]
Let $\X$ be a quantum set. Then $E\in\qRel(\X,\X)$ is called an \emph{equivalence relation} if \begin{itemize}
\item[(1)] $I_\X\leq E$
    \item[(2)] $E\circ E\leq E$;
    \item[(3)] $E^\dag=E$.
    \end{itemize}
\end{definition}

\begin{lemma}\label{lem:equivalence relations generated C*-subalgebras}
Let $E$ be an equivalence relation on a quantum set $\X$, and let $X\atomof\X$. Then $E(X,X)$ is a unital C*-subalgebra of $L(X):=L(X,X)$.
\end{lemma}
\begin{proof}
The subspace $E(X,X)$ is a unital subalgebra of $L(X)$ by Lemma \ref{lem:orders generated antisymmetric subalgebras}, because $E$ is a pre-order. Condition (3) specializes to the equation $E(X,X)^\dagger = E(X,X)$, so $E(X,X)$ is furthermore a $*$-subalgebra. It is automatically closed in the norm topology because $L(X)$ is finite-dimensional.
\end{proof}

\begin{lemma}\label{lem:pre-order generates equivalence relation}
Let $\X$ be a quantum set and let $R\in\qRel(\X,\X)$ be a pre-order. Then $E=R\wedge R^\dag$ is an equivalence relation. 
\end{lemma}
\begin{proof}
By Lemma \ref{lem:basic facts on quantum orders}, $R^\dag$ is a pre-order, too, so we have both $I_\X\leq R$ and $I_\X\leq R^\dag,$ hence $I_\X\leq R\wedge R^\dag=E$.
Furthermore, we have
\begin{align*} E\circ E & =(R\wedge R^\dag)\circ (R\wedge R^\dag)\leq (R\circ (R\wedge R^\dag))\wedge (R^\dag\circ(R\wedge R^\dag))\\
& \leq (R\circ R)\wedge (R\circ R^\dag)\wedge (R^\dag\circ R)\wedge (R^\dag\circ R^\dag)\\
& \leq (R\circ R)\wedge (R^\dag\wedge R^\dag)\leq R\wedge R^\dag =E,
\end{align*}
where we used Lemma \ref{lem:intersection} for the first two inequalities, and the fact that both $R$ and $R^\dag$ are pre-orders in the last inequality. 
Finally, using Lemma \ref{lem:action of dagger}, we obtain  
\[  E^\dag=(R\wedge R^\dag)^\dag=R^\dag\wedge R=E. \qedhere\]
\end{proof}

The following lemma is essentially the dual of \cite[Lemma 8.5]{Kornell2} in $\qSet$.

\begin{lemma}\label{lem:equivalence relation lemma}
Let $\X$ be a quantum set, and let $E\in\qRel(\X,\X)$ be an equivalence relation such that $E\neq I_\X$. Then there exists a function $G:\X\to\X$ such that $G\neq I_\X$, $G\leq E$, $G\circ G=I_\X$ and $G^\dag=G^{-1}=G$. 
\end{lemma}
\begin{proof}
We consider two cases. The first is $E(X,X)=I_\X(X,X)$ for each $X\atomof\X$; the second is $E(X_1,X_1)\neq I_\X(X_1,X_1)$ for some $X_1\in\X$. We consider the latter case first. By \ref{lem:equivalence relations generated C*-subalgebras}, $E(X_1,X_1)$ us a unital C*-subalgebra of $L(X_1, X_1)$, and by assumption, it is larger than $I_\X(X_1,X_1)=\CC 1_{X_1}$. Thus, it contains a non-trivial projection $p:X_1\to X_1$. Hence, it also contains $u:=1-2p$, which is a unitary $X_1\to X_1$ that satisfies $u=u^\dag=u^{-1}$. Now define $G\in\qRel(\X,\X)$ by 
\[ G(X,Y) = \begin{cases} \CC u, & X=Y=X_1;\\
\CC 1_X, & X=Y\neq X_1,\\
0, & X\neq Y.
\end{cases}\]
Clearly $G(X,Y)=G(Y,X)^\dag=G^\dag(X,Y)$ for all $X, Y \atomof \X$, so $G=G^\dag$. Moreover, we have $G(X_1,X_1)\leq E(X_1,X_1)$ by construction, hence $G\leq E$. A direct calculation yields $(G\circ G)(X,X)=\CC 1_X$ for all $X \atomof \X$, with $G(X,Y)$ vanishing otherwise; hence $G^2=I_\X$.

In the other case, since $E \neq I_\X$ but $E(X,X) = I_\X(X,X)$ for all $X \atomof \X$, there exist distinct $X_1,X_2\atomof\X$ such that $E(X_1,X_2) \neq I_\X(X_1, X_2) = 0$. Hence, let $a\in E(X_1,X_2)$ be non-zero. Its adjoint $a^\dag$ is in $ E(X_1,X_2)^\dag=E^\dag(X_2,X_1)=E(X_2,X_1)$. We compute that
\begin{align*}
  a^\dag a & \in E(X_2,X_1)\cdot E(X_1,X_2)\leq\bigvee_{Y\atomof\X} E(Y,X_1)\cdot E(X_1,Y)\\
  & =(E\circ E)(X_1,X_1)\leq E(X_1,X_1)=I(X_1,X_1)=\CC 1_{X_1},
  \end{align*}
  and similarly, $aa^\dag\in\CC 1_{X_2}$. Let $u=a/\|a\|$; then it follows that $u^\dag u=1_{X_1}$ and $uu^\dag=1_{X_2}$.
  Now define $G\in\qRel(\X,\X)$ by
  \[ G(X,Y) = \begin{cases} \CC u, & X=X_1, Y=X_2;\\
  \CC u^\dag, & X=X_2,Y=X_1;\\
\CC 1_X, & X_1\neq X=Y\neq X_2,\\
0, & \text{otherwise}.
\end{cases}\]
Here too, we find that $G^\dag(X,Y)=G(Y,X)^\dag=G(X,Y)$, so $G=G^\dag$, and $G(X,Y)\leq E(X,Y)$ by construction. 
A direct calculation yields $(G\circ G)(X,X)=\CC 1_X$ for all $X \atomof \X$, with $G(X,Y)$ vanishing otherwise; so here too, $G^2=I_\X$.

We conclude that in both cases we have $G\leq E$, and $G=G^\dag$, and $G\circ G=I_\X$. 
Hence
\[ G^\dag\circ G=G\circ G=I_\X=G\circ G=G\circ G^\dag,\]
so $G:\X\to\X$ is a function. 
\end{proof}

\begin{theorem}\label{thm:qPos is complete}
The category $\qPOS$ is complete. More specifically, given a diagram of shape $A$ consisting of objects $\X_\alpha$ for $\alpha\in A$ and monotone maps $F_f:\X_\alpha\to\X_\beta$ for each morphism $f:\alpha\to\beta$, the limit $(\X,R)$ in $\qPOS$ consists of the limit $\X$ in $\qSet$, and 
\[R=\bigwedge_{\alpha\in A}F^\dag_\alpha\circ R_\alpha\circ F_\alpha,\]
where the functions $F_\alpha:\X\to\X_\alpha$ are the limiting functions of this diagram in $\qSet$.
\end{theorem}

\begin{proof}
Firstly, we note that \cite[Proposition 8.7]{Kornell2} ensures that the diagram of shape $A$ in the statement indeed has a limit $\X$ in $\qSet$. Let $F_\alpha:\X\to\X_\alpha$ be the limiting functions, i.e., such that $F_f\circ F_\alpha=F_\beta$ for each function $f:\alpha\to \beta$ in $A$.

Let $E = \bigwedge_{\alpha\in A}F_\alpha^\dag\circ F_\alpha$; we claim that $E = I_\X$. Assume otherwise. The binary relation $E$ is a pre-order on $\X$ by Lemma \ref{lem:intersection of preorders is a preorder}. Furthermore, $E$ clearly satisfies $E^\dagger = E$, and it is therefore an equivalence relation. We apply Lemma \ref{lem:equivalence relation lemma} to obtain a function $G:\X\to\X$ such that $G\neq I_\X$ and $G\leq E$, and we calculate that $F_\alpha\circ G\leq F \circ E \leq F_\alpha\circ F_\alpha^\dag\circ F_\alpha\leq I_\X\circ F_\alpha=F_\alpha$ for each $\alpha \in A$. We conclude by Lemma \ref{lem:inequality between functions is equality} that $F_\alpha \circ G = F_\alpha$. This equality holds for each $\alpha\in A$, and hence $G=I_\X$, by the universal property of the limit in $\qSet$, contradicting our choice of $G$. Therefore, $E = I$.

Similarly, let $R = \bigwedge_{\alpha\in A}F_\alpha^\dag\circ R_\alpha \circ F_\alpha$; we claim that $R$ is an order on $\X$. Indeed, it is a pre-order by Lemma \ref{lem:intersection of preorders is a preorder}, and it is antisymmetric by the following calculation:
\begin{align*}
    R\wedge R^\dag & = \bigwedge_{\alpha,\beta\in A}(F_\alpha^\dag\circ R_\alpha\circ F_\alpha)\wedge (F_\beta^\dag\circ R_\beta\circ F_\beta)  \leq \bigwedge_{\alpha\in A}(F_\alpha^\dag\circ R_\alpha\circ F_\alpha)\wedge (F_\alpha^\dag\circ R_\alpha\circ F_\alpha)\\
    & = \bigwedge_{\alpha\in A}F^\dag_\alpha\circ (R_\alpha\wedge R_{\alpha}^\dag)\circ F_\alpha = \bigwedge_{\alpha\in A}F_\alpha^\dag\circ F_\alpha =E = I_\X
    \end{align*}
We have used Proposition \ref{prop:intersection} in the second equality. We conclude that $R$ is indeed an order on $\X$.

The definition of $R$ trivially implies that $R\leq F_\alpha^\dag\circ R_\alpha\circ F_\alpha,$ which expresses that $F_\alpha$ is monotone. Thus, we have a cone on the given diagram in $\qPOS$, and it remains only to show that it is a limiting cone.

Let $(\Y,T)$ be a quantum poset, and let $C_\alpha:\Y\to\X_\alpha$, for $\alpha\in A$, be monotone maps that together form a cone. Since $\X$ is the limit of the $\X_\alpha$ in $\qSet$, it follows that there is a unique function $H:\Y\to\X$ such that $C_\alpha=F_\alpha\circ H.$ By the monotonicity of $C_\alpha$, we have 
$T\leq C_\alpha^\dag\circ R_\alpha\circ C_\alpha=(F_\alpha\circ H)^\dag\circ R_\alpha\circ F_\alpha\circ H=H^\dag\circ F_\alpha^\dag\circ R_\alpha\circ F_\alpha\circ H.$
Hence,
$T\leq\bigwedge_{\alpha\in A}H^\dag\circ F_\alpha^\dag\circ R_\alpha\circ F_\alpha\circ H=H^\dag\circ \left(\bigwedge_{\alpha\in A}F_\alpha^\dag\circ R_\alpha\circ F_\alpha\right)\circ H=H^\dag\circ R\circ H,$
where we have used Proposition \ref{prop:intersection} in the penultimate equality. Hence, $H$ is monotone. We have thus established that the functions $F_\alpha$, for $\alpha \in A$, together form a limiting cone and that, more generally, the category $\qPOS$ is complete.
\end{proof}

\section{Cocompleteness}
We show that the category $\cat{qPOS}$ is cocomplete. Unlike limits in $\cat{qPOS}$, colimits in $\cat{qPOS}$ cannot be formed simply by ordering the corresponding colimits in $\cat{qSet}$. However, coproducts in $\cat{qPOS}$ are simply coproducts in $\cat{qSet}$, ordered appropriately, and we begin with this special case.

The categories $\qRel$ and $\qSet$ are cocomplete \cite{Kornell2}. The coproduct of an indexed family $\{\X_\alpha\}_{\alpha \in A}$ of quantum sets is the same in both categories, and it is easiest to characterize when the quantum sets $\X_\alpha$ have no atoms in common. In this special case, the coproduct $\biguplus_{\alpha \in A} \X_\alpha$ in $\cat{qRel}$ and $\cat{qSet}$ is simply their union $\bigcup_{\alpha \in A} \X_\alpha$, which is obtained by collecting the atoms of all of the quantum sets $\X_\alpha$ with $\alpha \in A$. Then, each quantum set $\X_\alpha$ is a subset of $\biguplus_{\alpha \in A} \X_\alpha$, and its inclusion function $J_{\X_\alpha}\: \X_\alpha \to\biguplus_{\alpha \in A} \X_\alpha$ is the corresponding edge of the colimiting cocone. This defines the coproduct of an arbitrary family up to isomorphism. In the general case, some of the quantum sets $\X_\alpha$ may have atoms in common, but we may replace these quantum sets by (disjoint) isomorphic quantum sets to avoid this nuisance.

If $\{R_\alpha\}_{\alpha \in A}$ is a family of binary relations, with $R_\alpha$ being a binary relation from $\X_\alpha$ to $\Y$, then $[R_\alpha: \alpha \in A]$ denotes the unique binary relation from $\biguplus_{\alpha \in A} \X_\alpha$ to $\Y$ that comes from the universal property of $\biguplus_{\alpha \in A} \X_\alpha$. Also, let $\bot_\X^\Y$ denote the minimum binary relation from $\X$ to $\Y$. We will often leave the subscript and the superscript implicit, simply writing $\bot$.

\begin{lemma}\label{lem:coproduct order}
Let $\{\X_\alpha\}_{\alpha\in A}$ and $\{\Y_\beta\}_{\beta\in B}$ be collections of quantum sets, and let $\X=\biguplus_{\alpha \in A} \X_{\alpha}$ and $\Y=\biguplus_{\beta\in B}\Y_\beta$ be their coproducts. Let $J_\alpha:\X_\alpha\to\X$ and $K_\beta:\Y_\beta\to\Y$ be the canonical injections. Then all of the following hold:
\begin{itemize}
    \item[(a)] For all $\alpha \in A$, $J_\alpha^\dag\circ J_\alpha=I_{\X_\alpha}$, and for all distinct $\alpha_1, \alpha_2 \in A$, $J_{\alpha_1}^\dag \circ J_{\alpha_2} = \bot$.
    \item[(b)] Let $R,S\in\qRel(\X,\Y)$. Then, the following are equivalent:
\begin{itemize}
    \item[(1)] $R\leq S$,
    \item[(2)] $R\circ J_\alpha\leq S\circ J_\alpha$ for all $\alpha\in A$,
    \item[(3)] $K_\beta^\dag\circ R\leq K_\beta^\dag\circ S$ for all $\beta\in B$, and
    \item[(4)] $K_\beta^\dag\circ R\circ J_\alpha\leq K_\beta^\dag\circ S\circ J_\alpha$ for all $\alpha\in A$ and $\beta\in B$.
\end{itemize}  
    \item[(c)] For each $\alpha\in A$ and $n \in \NN$, let $T_{\alpha,n} \in \qRel(\X_\alpha,\Y)$. Then,
    \[ \left[\bigwedge_{n\in\NN}T_{\alpha,n}:\alpha\in A\right]=\bigwedge_{n\in\NN}[T_{\alpha,n}:\alpha\in A].\]
    \item[(d)] Assume $A = B$. For each $\alpha\in A$ and $n \in \NN$, let $T_{\alpha,n} \in \qRel(\X_\alpha,\Y_\alpha)$. Then,
    \[\biguplus_{\alpha\in A}\left(\bigwedge_{n\in\NN}
T_{\alpha,n}\right)=\bigwedge_{n\in\NN}\left(\biguplus_{\alpha\in A} T_{\alpha,n}\right).\]
 
\end{itemize}
\end{lemma}
\begin{proof}
Without loss of generality, we can assume that the quantum sets $\X_\alpha$ are pairwise disjoint; hence, $J_\alpha=J_{\X_\alpha}^\X$ \cite{Kornell2}*{Definition 8.2}. Similarly, we can assume that the quantum sets $\Y_\beta$ are pairwise disjoint; hence, $K_\beta=K_{\Y_\beta}^\Y$. Then, (a) follows from a direct calculation.

For (b), it is clear that (1) implies (2) and (3). It is also clear both that (2) implies (4) and that (3) implies (4). So we only have to show that (4) implies (1). Assume (4), and fix $X\atomof\X$ and $Y\atomof\Y$. The Hilbert space $X$ is an atom of $\X_\alpha$ for some $\alpha \in A$, and similarly, the Hilbert space $Y$ is an atom of $\Y_\beta$ for some $\beta \in B$. Applying Lemma \ref{lem:restriction and corestriction}, we compute that
\[ R(X,Y)=(K_\beta^\dag\circ R\circ J_\alpha)(X,Y)\leq (K_\beta^\dag\circ S\circ J_\alpha)(X,Y)=S(X,Y).\]
We now vary $X \atomof \X$ and $Y \atomof \Y$ to conclude that $R \leq S$, i.e., to conclude (1).

For (c), we compute that for each $\alpha_0 \in A$,
 \[ \left[\bigwedge_{n\in\NN}T_{\alpha,n}:\alpha\in A\right]\circ J_{\alpha_0}=\bigwedge_{n\in\NN}T_{{\alpha_0},n}=\bigwedge_{n\in\NN}([T_{\alpha,n}:\alpha\in A]\circ J_{\alpha_0})=\left(\bigwedge_{n\in\NN}[T_{\alpha,n}:\alpha\in A]\right)\circ J_{\alpha_0},\]
 where we apply Proposition \ref{prop:intersection} for the last equality. Claim (c) now follows from (b).
 
 For (d), we compute that for each $\alpha_0 \in A$:
 
 \[\left(\biguplus_{\alpha\in A}\left(\bigwedge_{n\in\NN}
T_{\alpha,n}\right)\right)\circ J_{{\alpha_0}}=\bigwedge_{n\in\NN}T_{{\alpha_0},n}=\bigwedge_{n\in\NN}\left(\left(\biguplus_{\alpha\in A}T_{\alpha,n}\right)\circ J_{\alpha_0}\right)=\left(\bigwedge_{n\in\NN}\left(\biguplus_{\alpha\in A} T_{\alpha,n}\right)\right)\circ J_{\alpha_0},\]
 where we used Proposition \ref{prop:intersection} in the last equality. Claim (d) now follows from (b).
\end{proof}

\begin{proposition}\label{prop:coproducts in qPos}
Let $\{(\X_\alpha,R_\alpha)\}_{\alpha\in A}$ be a collection of quantum posets. Let $\X=\biguplus_{\alpha\in A}\X_\alpha$ be the coproduct of the $\X_\alpha$ in $\qSet$, and let $R=\biguplus_{\alpha\in A}R_\alpha$ be the coproduct of the $R_\alpha$ as morphisms in $\qRel$, i.e., the unique $R\in\qRel(\X,\X)$ such that, for each $\alpha\in A$,
\begin{equation}\label{eq:injections coproduct in qPos}
R\circ J_\alpha=J_\alpha\circ R_\alpha.    
\end{equation}
Then $(\X,R)$ is the coproduct of the $(\X_\alpha,R_\alpha)$ in $\qPOS$, and the canonical injections $J_\alpha:\X_\alpha\to\X$ are order embeddings.
\end{proposition}
\begin{proof}
For any binary relation $R$ from a quantum set $\X$ to a quantum set $\Y$, we write $\Delta_{\alpha, \beta}R$ to mean $R$ if $\alpha = \beta$ and $\bot_\X^\Y$ if $\alpha \neq \beta$. Formally, this notation can be understood in terms of scalar multiplication, which can be defined in any dagger compact category \cite{AbramskyCoecke}*{section 3.2}.

Fix $\alpha,\beta\in A$. By Lemma \ref{lem:coproduct order}, we have $J_\beta^\dag\circ J_\alpha=\Delta_{\alpha,\beta}I_{\X_\alpha}$; hence, using that $R_\alpha$ is an order, we compute that
\begin{align*}
J_\beta^\dag\circ I_\X\circ J_\alpha & =J_\beta^\dag\circ J_\alpha\circ I_{\X_\alpha}\leq J_\beta^\dag\circ J_\alpha\circ R_\alpha =J_\beta^\dag\circ R\circ J_\alpha,\\
J_\beta^\dag\circ R\circ R\circ J_\alpha & = J_\beta^\dag\circ R\circ J_\alpha\circ R_\alpha = J_\beta^\dag\circ J_\alpha\circ R_\alpha\circ R_\alpha\leq J_\beta^\dag\circ J_\alpha\circ R_\alpha=J_\beta^\dag\circ R\circ J_\alpha,\\
J_\beta^\dag\circ (R\wedge R^\dag)\circ J_\alpha & = (J_\beta^\dag\circ R\circ J_\alpha)\wedge(J_\beta^\dag\circ R^\dag \circ J_\alpha)= (J_\beta^\dag\circ R\circ J_\alpha)\wedge((R\circ J_\beta)^\dag\ \circ J_\alpha)\\
& = (J_\beta^\dag\circ J_\alpha\circ R_\alpha)\wedge ((J_\beta\circ R_\beta)^\dag\circ J_\alpha)=(J_\beta^\dag\circ J_\alpha\circ R_\alpha)\wedge (R_\beta^\dag\circ J_\beta^\dag\circ J_\alpha)\\
& = (\Delta_{\alpha,\beta} I_{\X_\alpha}\circ R_\alpha)\wedge(R_\beta^\dag\circ\Delta_{\alpha,\beta}I_{\X_\alpha})=\Delta_{\alpha,\beta}(R_\alpha\wedge R_\alpha^\dag)
\\ &
=\Delta_{\alpha,\beta}I_{\X_\alpha}=J_\beta^\dag\circ I_\X\circ J_\alpha, 
\end{align*}
We have used Proposition \ref{prop:intersection} for the calculation of $J_\beta^\dag\circ (R\wedge R^\dag)\circ J_\alpha$. Since these (in)equalities hold for each $\alpha,\beta\in A$, it follows from Lemma \ref{lem:coproduct order} that $I_\X\leq R$, $R\circ R\leq R$, and $R\wedge R^\dag=I_\X$, i.e., that $R$ is an order on $\X$. Now, we find from Equation (\ref{eq:injections coproduct in qPos}) and Lemma \ref{lem:coproduct order} that $J_\alpha^\dag\circ R\circ J_\alpha=J_\alpha^\dag\circ J_\alpha=I_{\X_\alpha}$;
hence, $J_\alpha$ is an order embedding. 

Let $(\Y,S)$ be a quantum poset, and let $F_\alpha:\X_\alpha\to\Y$ be a collection of monotone maps. We need to check that $[F_\alpha:\alpha\in A]:\X\to\Y$ is monotone, too. For each $\beta\in A$, we have $[F_\alpha:\alpha\in A]\circ R\circ J_\beta = [F_\alpha:\alpha\in A]\circ J_\beta\circ R_\beta=F_\beta\circ R_\beta\leq S\circ F_\beta=S\circ [F_\alpha:\alpha\in A]\circ J_\beta,$
where we use Equation (\ref{eq:injections coproduct in qPos}) in the first and the last equality, whereas the inequality follows by the monotonicity of $F_\beta$. Since the resulting inequality holds for each $\beta\in A$, it follows from Lemma \ref{lem:coproduct order} that $[F_\alpha:\alpha\in A]\circ R\leq S\circ [F_\alpha:\alpha\in A]$. Therefore, $[F_\alpha:\alpha\in A]$ is indeed monotone, which concludes the proof that $(\X,R)=\biguplus_{\alpha\in A}(\X_\alpha,R_\alpha)$ in $\qPOS$.
\end{proof}

\begin{proposition}\label{prop:coproduct is POS-enriched}
Let $\{\X_\alpha\}_{\alpha\in A}$ be a collection of quantum sets, let $\{\Y_\alpha,S_\alpha\}$ be a collection of quantum posets, let $\X=\biguplus_{\alpha\in A}\X_\alpha$ in $\qSet$, and let $(\Y,S)=\biguplus_{\alpha}(\Y_\alpha,S_\alpha)$ in $\qPOS$. For each $\alpha\in A$, let $F_\alpha$ and $G_\alpha$ be functions from $\X_\alpha$ to $\Y_\alpha$ such that $F_\alpha\sqsubseteq G_\alpha$. Let $F= \biguplus_{\alpha\in A}F_\alpha$ and $G= \biguplus_{\alpha\in A}G_\alpha$ be functions from $\X$ to $\Y$. Then, $F\sqsubseteq G$.
\end{proposition}
\begin{proof}
Since $F_\alpha\sqsubseteq G_\alpha$, we have $S_\alpha\circ G_\alpha\leq S_\alpha\circ F_\alpha$. For each $\beta\in A$, let $J_\beta:\X_\beta\to\X$ be the canonical injection. We find that
\begin{align*}
    S\circ G\circ J_\beta=\left(\biguplus_{\alpha\in A}S_\alpha\circ G_\alpha\right)\circ J_\beta=S_\beta\circ G_\beta\leq S_\beta\circ F_\beta =\left(\biguplus_{\alpha\in A}S_\alpha\circ F_\alpha\right)\circ J_\beta=S\circ F\circ J_\beta.
\end{align*}
Hence, it follows by Lemma \ref{lem:coproduct order} that $S\circ G\leq S\circ F$, i.e., $F\sqsubseteq G$.
\end{proof}

Next, we aim to show that $\qPOS$ is cocomplete. Since we have already shown that it has all coproducts, it is sufficient to show that is has coequalizers. But it is difficult to give an explicit description of coequalizers in $\cat{qSet}$, so we choose instead to follow a more abstract route: recall that a category is \emph{well-powered} if the subobjects of any object form a set, \emph{co-well powered} if the quotient objects of any object form a set, and \emph{extremally co-well powered} if those quotient objects that are represented by extremal epimorphisms form a set. Here, we recall that an epimorphism $E$ is said to be \emph{extremal} if $E=M\circ F$ for some monomorphism $M$ implies that $M$ is an isomorphism. For us, the importance of these concepts is that any category that is complete, well-powered and extremally co-well powered has all coequalizers. This theorem is originally proven in \cite[Theorem 5.11]{Nakagawa}, but also stated in \cite{AdamekHerrlichStrecker} as Exercise 12J.

\begin{theorem}\label{thm:qPOS is cocomplete}
The category $\qPOS$ is well-powered and extremally co-well powered. It is therefore cocomplete.
\end{theorem}

\begin{proof}
The category $\cat{W}_1^*$ of von Neumann algebras and unital normal $*$-homomorphisms is well-powered and co-wellpowered: the subjects of a von Neumann algebra correspond to its unital ultraweakly closed $*$-subalgebras, and the quotient objects of a von Neumann algebra correspond to its ultraweakly closed two-sided ideals. Hence, the full subcategory $\cat{M}_1^*$ of hereditarily atomic von Neumann algebras is also well-powered and co-well powered. Therefore, $\cat{qSet}$ is also well-powered and co-well powered \cite{Kornell2}*{Theorem 7.4}.

Let $(\X,R)$ be a quantum poset, and choose representatives for the subobjects of $\X$. A subobject of $\X$ is an equivalence class of pairs $(\W, M)$, with $\W$ a quantum set and $M$ a monomorphism $\W \To \X$. Two such pairs, $(\W_1, M_1)$ and $(\W_2, M_2)$, are defined to be equivalent if they are isomorphic as objects in the slice category $\cat{qSet}/\X$, i.e., if there is an isomorphism $F\: \W_1 \To \W_2$ such that $M_1 = M_2 \circ F$. We choose a family of pairs $\{(\W_\alpha,M_\alpha)\}_{\alpha \in \mathrm{Sub}(\X)}$ to represent these equivalence classes.

Let $((\W, T), M)$ represent a subobject of $(\X, R)$ in $\cat{qSet}$. Thus, $(\W, T)$ is a poset, and $M$ is a monomorphism from $(\W, T)$ to $(\X,R)$. By Lemma \ref{lem:monomorphisms in qPOS}, $M$ is an injection, i.e., a monomorphism in $\cat{qSet}$. By our choice of the pairs $(\W_\alpha,M_\alpha)$, for $\alpha \in \mathrm{Sub}(\X)$, there exist a subobject $\alpha \in \mathrm{Sub}(\X)$ and an isomorphism $F\: \W_\alpha \to \W$ such that $M_\alpha = M \circ F$. Thus, $F$ is an isomorphism in $\cat{qPOS}$ from the quantum poset $(\W_\alpha, F^\dag \circ T \circ F)$ to the quantum poset $(\W,T)$ such that $M_\alpha = M \circ F$. The function $M$ is monotone by the definition of a subobject, and the function $F$ is monotone because it is an isomorphism, so $M_\alpha$ is also monotone. Therefore, we have show that $((\W, T), M)$ is isomorphic to a pair of the form $((\W_\alpha, T'), M_\alpha)$ for some $\alpha \in \mathrm{Sub}(\X)$ and some order $T'$ on $\omega_\alpha$. Since $\cat{qSet}$ is well-powered, $\mathrm{Sub}(\X)$ is a set, and furthermore, for each $\alpha \in \mathrm{Sub}(\X)$, the orders on $\W_\alpha$ form a set. We conclude that the subobjects of $((\W, T), M)$ in $\cat{qPOS}$ form a set, and more generally, that $\cat{qPOS}$ is well-founded.

The proof that $\cat{qPOS}$ is extremally co-well powered is entirely similar, replacing monomorphisms into an arbitrary quantum poset $((\W, T), M)$ with extremal epimorphisms out of an arbitrary quantum poset $((\W, T), M)$, and appealing to Lemma \ref{lem:extremal epi in qPOS} instead of Lemma \ref{lem:monomorphisms in qPOS}. Hence, the category $\cat{qPOS}$ is both well-powered and extremally co-well powered. By Theorem \ref{thm:qPos is complete}, it is complete, so by \cite[Theorem 5.11]{Nakagawa}, it has all coequalizers. It also has all coproducts by Proposition \ref{prop:coproducts in qPos}. Therefore, $\cat{qPOS}$ is cocomplete.
\end{proof}

\section{Monoidal product}
We show that $\cat{qPOS}$ is a symmetric monoidal category, with its monoidal structure inherited from $\cat{qSet}$. This monoidal structure on $\cat{qSet}$, and more generally on $\qRel$, is termed the Cartesian product, and it is notated $\times$, but it is not the category-theoretic product. This terminology and notation are justified by the fact that this monoidal structure generalizes the familiar Cartesian product of ordinary sets in a sense appropriate to the noncommutative dictionary.

Let $\X$ and $\Y$ be quantum sets. Their Cartesian product $\X\times\Y$ is the quantum set defined by $\At(\X \times \Y)=\{X\otimes Y:X\in\X,Y\in\Y\}$. Furthermore, if $R$ is some binary relation on $\X$ and $S$ is some binary relation on $\Y$, then their product $R \times S$ is the binary relation on $\X \times \Y$ defined by 
$ (R\times S)(X_1\otimes Y_1,X_2\otimes Y_2) = R(X_1,X_2)\otimes S(Y_1,Y_2),$
for $X_1, X_2 \atomof \X$ and $Y_1, Y_2 \atomof \Y$.

\begin{proposition}\label{prop:monoidal product of quantum posets}
Let $(\X,R)$ and $(\Y,S)$ be quantum posets. The, $(\X\times \Y,R\times S)$ is also a quantum poset.
\end{proposition}
\begin{proof}
Verifying that $R\times S$ is a quantum pre-order is routine. Furthermore, we compute that $(R \times S) \wedge (R \times S)^\dagger = (R \times S) \wedge (R^\dag \times S^\dag) = (R \wedge R^\dag) \times (S \wedge S^\dag) = I_\X \times I_\Y = I_{\X, \Y}$. The key second equality may be checked atom by atom. Thus, the pre-order $R\times S$ is in fact an order on $\X \times \Y$.
\end{proof}

\begin{lemma}\label{lem:tensor product of monotone maps is monotone}
Let $(\X_1,R_1),(\X_2,R_2),(\Y_1,S_1),(\Y_2,S_2)$ be quantum posets, and let $F:(\X_1,R_1)\to(\X_2,R_2)$ and $G:(\Y_1,S_1)\to(\Y_2,S_2)$ be monotone. Then \[F\times G:(\X_1\times \Y_1,R_1\times S_1)\to(\X_2\times \Y_2,R_2\times S_2)\] is monotone.
\end{lemma}
\begin{proof}
\[(F\times G)\circ (R_1\times S_1)=(F\circ R_1)\times (G\circ S_1)\leq (R_2\circ F)\times (S_2\circ G)=(R_2\times S_2)\circ (F\times G)\]
\end{proof}

\begin{lemma}\label{lem:projections are monotone}
Let $(\X,R)$ and $(\Y,S)$ be quantum pre-orders. Let $P:\X\times \Y\to\X$ and $Q:\X\times\Y\to\Y$ be the projection functions, in the sense of Appendix \ref{sub:projection functions}. Then 
\[  P\circ (R\times S)= R\circ P, \qquad Q\circ (R\times S)=Q\circ S,\]
hence $P$ and $Q$ are monotone. 
\end{lemma}
\begin{proof}
The proof is by direct calculation, atom by atom.
\end{proof}

\begin{proposition}\label{prop:M monotone iff PM and QM monotone}
Let $(\X,R)$, $(\Y,S)$ and $(\Z,T)$ be quantum posets, and let $M:\Z\to\X\times\Y$ be a function. Then $M$ is monotone if and only if $P\circ M$ and $Q\circ M$ are both monotone.
\end{proposition}
\begin{proof}
Assume $M$ is monotone. By Lemma \ref{lem:projections are monotone}, both $P$ and $Q$ are monotone. It now follows from Lemma \ref{lem:composition of monotone functions is monotone} that $P\circ M$ and $Q\circ M$ are both monotone.
For the converse, assume that $P\circ M:\Z\to\X$ and $Q\circ M:\Z\to\Y$ are both monotone. By Definition \ref{def:monotone map} this means that
\begin{align*}
    P\circ M\circ T\circ M^\dag\circ P^\dag \leq R,\qquad 
        Q\circ M\circ T\circ M^\dag\circ Q^\dag \leq S.
\end{align*}
Applying Lemma \ref{lem:relations on tensor products} to the relation $M\circ T\circ M^\dag$ on $\X\times\Y$ yields
\[M\circ T\circ M^\dag\leq (P\circ M\circ T\circ M^\dag\circ P^\dag)\times(Q\circ M\circ T\circ M^\dag\circ Q^\dag)\leq R\times S,\]
which expresses that $M$ is indeed monotone.
\end{proof}

Recall $\mathbf 1$ denotes the quantum set whose only atom is $\CC$.

\begin{theorem}\label{thm:qPOS is monoidal}
$(\qPOS,\times,(\mathbf 1,I_{\mathbf 1}))$ is a symmetric monoidal category.
\end{theorem}
\begin{proof}
By Proposition \ref{prop:monoidal product of quantum posets} and Lemma \ref{lem:tensor product of monotone maps is monotone}, the monoidal product $\times$ on $\qSet$ induces a bifunctor $\times$ on $\qPOS$. The monoidal structure of $\qSet$ is inherited from the monoidal structure on $\qRel$ \cite{Kornell2}*{Theorem 3.6}. The statement follows now from verifying that the associator, the unitors and the braiding are order isomorphisms, which is routine.
\end{proof}

The monoidal product $\times$ on $\cat{qSet}$ is not the category-theoretic product. It generalizes the familiar Cartesian product of ordinary sets in another sense that is appropriate to quantum physics and noncommutative geometry. For example, the quantum set $`\RR \times `\RR$ models pairs of real numbers, which cannot be the product of $`\RR$ with itself in the category of quantum sets and functions because functions into $`\RR$ model observables. Most pairs of observables are not compatible, so most pairs of functions into $`\RR$ should not correspond to a function into $`\RR \times `\RR$. For functions $F$ and $G$ from a quantum set $\Z$ to $`\RR$, we write $(F,G)$ for the corresponding function from $\Z$ to $`\RR \times `\RR$, \emph{if such a function exists}. More generally, for each function $F\: \Z \To \X$ and each function $G\: \Z \To \Y$, we write $(F, G)$ for the unique function $\Z \To \X \times \Y$ such that $P \circ (F, G) = F$ and $Q \circ (F,G) = G$ \emph{if such a function exists}, where $P\: \X \times \Y \to \X$ and $Q\: \X \times \Y \to \Y$ are the canonical projection functions \cite{Kornell2}*{Sections I.B and X}. In this case, we say that $F$ and $G$ are \emph{compatible}.

\begin{proposition}\label{prop:product maps are order embeddings}
Let $(\X,R)$ and $(\Y,S)$ be quantum posets, and let $\Z$ be a quantum set. Let $F_1,F_2:\Z\to\X$ and $G_1,G_2:\Z\to\Y$ be functions. If $F_1$ is compatible with $G_1$ and $F_2$ is compatible with $G_2$, then
\[ (F_1,G_1)\sqsubseteq (F_2,G_2) \text{ if and only if both }F_1\sqsubseteq F_2\text{ and }G_2\sqsubseteq G_2.\]
\end{proposition}
\begin{proof}
Assume $(F_1,G_1)\sqsubseteq (F_2,G_2)$. Since $P$ and $Q$ are monotone by Lemma \ref{lem:projections are monotone}, it follows that
$F_1=P\circ (F_1,G_1)\sqsubseteq P\circ (F_2,G_2)=F_2$ and $G_1=Q\circ (F_1,G_1)\sqsubseteq Q\circ (F_2,G_2)=G_2.$
Now assume that $F_1\sqsubseteq F_2$ and $G_1\sqsubseteq G_2$. This means that $F_2\circ F_1^\dag\leq R$ and $G_2\circ G_1^\dag\leq S$; hence, by Proposition \ref{prop:order between product maps} it follows that $(F_2,G_2)\leq ((F_2\circ F_1^\dag)\times (G_2\circ G_1^\dag))\circ (F_1,G_1)\leq (R\times S)\circ(F_1,G_1),$
which expresses that $(F_1,G_1)\sqsubseteq (F_2,G_2).$
\end{proof}

\begin{corollary}\label{cor:tensor product is an order embedding}
Let $(\X,R)$ and $(\Y,S)$ be quantum posets, and let $\V,\W$ be quantum sets. Let $F_1,F_2:\V\to\X$ and $G_1,G_2:\W\to\Y$ be functions. Then $F_1\times G_1 \sqsubseteq F_2\times G_2$  if and only if $F_1\sqsubseteq F_2$ and  $G_1\sqsubseteq G_2$. 
\end{corollary}
\begin{proof}
Consider the following diagram for $i=1,2$: \[
\begin{tikzcd}
\V\arrow{d}[swap]{F_i} &
\V\times\W \arrow{l}[swap]{P} \arrow{r}{P}
\arrow{d}{F_i\times G_i}
& \W\arrow{d}{G_i}
\\
\X
&
\X \times \Y \arrow{l}{P} \arrow{r}[swap]{Q}
&
\Y.
\end{tikzcd}\]
We first show that $F_1\sqsubseteq F_2$ is equivalent to $F_1\circ P\sqsubseteq F_2\circ P$. If $F_1\sqsubseteq F_2$, then $F_1\circ P\sqsubseteq F_2\circ P$ by Lemma \ref{lem:right multiplication is monotone}. Conversely, if $F_1\circ P\sqsubseteq F_2\circ P$, then $ F_2\circ P\circ (F_1\circ P)^\dag\leq R$;
hence $F_2\circ F_1^\dag=F_2\circ P\circ P^\dag\circ F_1^\dag =F_2\circ P\circ (F_1\circ P)^\dag\leq R,$
where we appeal to the surjectivity of $P$ (cf. Lemma \ref{lem:projection maps are surjective}) in the third equality. Thus $F_1\sqsubseteq F_2$. In a similar way we find that $G_1\sqsubseteq G_2$ if and only if $G_1\circ Q\sqsubseteq G_2\circ Q$.

Since for each $i=1,2$, we have $F_i\times G_i=(F_i\circ P,G_i\circ Q),$
it follows from Proposition \ref{prop:product maps are order embeddings} that $F_1\times G_1\sqsubseteq F_2\times G_2$ if and only if $F_1\circ P\sqsubseteq F_2\circ P$ and $G_1\circ Q\sqsubseteq G_2\circ Q$.
\end{proof}

\section{Monoidal closure}

We show that the category $\cat{qPOS}$ is monoidal closed. Let $(\Y,S)$ and $(\Z,T)$ be quantum posets. Intuitively, we construct the hom quantum poset $[(\Y,S), (\Z,T)]_\below$ by taking the largest subset of the quantum function set $\Z^\Y$ that consists of monotone functions and ordering them pointwise.

\begin{definition}
Let $(\X,R)$ and $(\Y,S)$ be quantum sets equipped with binary relations. Then we say that a function $F:\X\to\Y$ is a \emph{homomorphism} $(\X,R)\to(\Y,S)$ if $F\circ R\leq S\circ F$.
\end{definition}
\noindent Note that if $R$ and $S$ are orders on $\X$ and $\Y$, respectively, then $F$ is a homomorphism if and only if $F$ is monotone.

For the next lemma, recall that $\qSet$ is monoidal closed \cite{Kornell2}*{Theorem 9.1}, so for each pair of quantum sets $\Y$ and $\Z$, there is an exponential object $\Z^\Y$, and the monoidal product $\times$ has a right adjoint. Thus, $\times$ preserves colimits, in particular the monoidal product distributes over coproducts; this fact will be used in the proof of the next lemma. For each subset $\W \subseteq \Z^\Y$, we write $J_\W\: \W \hookrightarrow \Z^\Y$ for the canonical inclusion.

\begin{lemma}\label{lem:monoidal closure lemma}
Let $(\Y,S)$ and $(\Z,T)$ be quantum sets equipped with binary relations. Then, there exists a quantum set $\W \subsetof \Z^\Y$ that is the largest subset of $\Z^\Y$ such that
\[\Eval\circ (J_\W\times I_\Y):(\W\times \Y,I_\W\times S)\to(\Z,T)\]
is a homomorphism, and there exists a reflexive binary relation $Q$ on $\W$ that is the largest binary relation on $\W$ such that
\[\Eval\circ (J_\W\times I_\Y):(\W\times \Y,Q\times S)\to(\Z,T)\]
is a homomorphism. Moreover, $\W$ and $Q$ satisfy and are uniquely determined by the following properties:
\begin{enumerate}
\item The function $\mathrm{Eval} \circ (J_\W \times I_\Y)$ is a homomorphism $(\W \times \Y, Q \times S) \To (\Z, T)$.
\item For every quantum set $\X$ equipped with a reflexive binary relation $R$, and every homomorphism $F\: (\X \times \Y, R \times S) \To (\Z, T)$, there exists a unique homomorphism $G\: (\X,R) \To (\W, Q)$ such that $F = \mathrm{Eval} \circ (J_\W \times I_\Y) \circ (G \times I_\Y)$:
$$
\begin{tikzcd}
(\X \times \Y, R \times S)
\arrow{rrrd}{F}
\arrow[dotted]{d}[swap]{G \times I_\Y}
&
&
&
\\
(\W \times \Y, Q \times S)
\arrow{rrr}[swap]{\mathrm{Eval} \circ (J_\W \times I_\Y)}
&
&
&
(\Z, T)
\end{tikzcd}
$$
\end{enumerate}
\end{lemma}

\begin{proof}
To show that $\W$ is well defined, we show that the set $\fF$ of subsets $\V \subsetof \Z^\Y$ such that $\Eval\circ (J_\V \times I_\Y)$ is a homomorphism from $(\V \times \Y, I_\V \times S)$ to $(\Z, T)$ is closed under subsets and arbitrary disjoint unions. For all $\V \in \fF$, and all $\V' \subsetof \V$, we write $J_{\V'}^{\V}$ for the canonical inclusion of $\V'$ into $\V$, and  we calculate that
\begin{align*}
\Eval \circ (J_{\V'} \times I_\Y) \circ (I_{\V'} \times S)
&=
\Eval \circ (J_{\V} \times I_\Y) \circ (J_{\V'}^{\V} \times I_\Y) \circ (I_{\V'} \times S)
\\ &=
\Eval \circ (J_{\V} \times I_\Y) \circ (I_{\V} \times S) \circ (J_{\V'}^{\V} \times I_\Y)
\\ & \leq
T \circ \Eval \circ (J_{\V} \times I_\Y) \circ (J_{\V'}^{\V} \times I_\Y)
\\ & =
T \circ \Eval \circ (J_{\V'} \times I_\Y).
\end{align*}
Thus, $\V' \in \fF$. We conclude that $\fF$ is closed under subsets. For all $\V_1, \V_2 \in \fF$, if $\V_1$ and $\V_2$ are disjoint in the sense that they have no atoms in common, then we make the identification $\V_1 \uplus \V_2 = \V_1 \union \V_2$, and we calculate that
\begin{align*}
\Eval \circ (J_{\V_1 \union \V_2} \times I_\Y) & \circ (I_{\V_1 \union \V_2} \times S)
=
\Eval \circ (J_{\V_1 \union \V_2} \times S)
=
\Eval \circ ([J_{\V_1}, J_{\V_2}] \times S)
\\ &=
\Eval \circ [J_{\V_1} \times S, J_{\V_2} \times S]
=
[\Eval \circ (J_{\V_1} \times S), \Eval \circ (J_{\V_2} \times S)]
\\ &=
[\Eval \circ (J_{\V_1} \times I_\Y) \circ (I_{\V_1} \times S), \Eval \circ (J_{\V_2} \times I_\Y) \circ (I_{\V_2} \times S)]
\\ &\leq
[T \circ \Eval \circ (J_{\V_1} \times I_\Y) , T \circ \Eval \circ (J_{\V_2} \times I_\Y) ]
\\ &=
T \circ \Eval \circ [J_{\V_1} \times I_\Y, J_{\V_2} \times I_\Y]
\\ &=
T \circ \Eval \circ ([J_{\V_1}, J_{\V_2}] \times I_\Y)
=
T \circ \Eval \circ (J_{\V_1 \union \V_2} \times I_\Y).
\end{align*}
Apart from a single step in which we apply our assumption that $\V_1, \V_2 \in \fF$, the entire calculation is category-theoretic. It is simplified by the fact that $(\V_1 \union \V_2) \times \Y = (\V_1 \times \Y) \union (\V_2 \times \Y)$. The same argument applies to arbitrarily large families of pairwise disjoint elements of $\fF$. We conclude that $\fF$ is closed under arbitrary disjoint unions. Therefore, $\fF$ has a largest subset, i.e., $\W$ is well defined.

We construct the binary relation $Q$ on $\W$ directly as the join of all the binary relations on $\W$ satisfying the desired condition. For each binary relation $Q'$ on $\W$, the function $\Eval\circ (J_\W \times I_\Y)$ is a homomorphism from $(\W \times \Y, Q' \times S)$ to $(\Z, T)$ if and only if 
$$ \Eval\circ (J_\W \times I_\Y) \circ (Q' \times S) \leq T \circ \Eval\circ (J_\W \times I_\Y),$$
simply by definition. The composition of binary relations respects arbitrary joins, so the join of all binary relations $Q'$ satisfying this inequality also satisfies this inequality. This is exactly the binary relation $Q$. It is reflexive because $\Eval \circ (J_\W \times I_\Y)$ is a homomorphism from $(\W \times \Y, I_\W \times S)$ to $(\Z, T)$, simply by the definition of $\W$.

Property (1) holds by construction. To establish property (2), let $F$ be a homomorphism from  $(\X \times \Y, R \times S)$ to $(\Z, T)$. The universal property of the quantum function set $\Z^\Y$ guarantees the existence of a function $G_0 \: \X \To  \Z^\Y$ satisfying $\Eval \circ ( G_0 \times I_\Y) = F$. The function $G_0$ factors through its range $\V_0 \subsetof \Z^\Y$, yielding a surjective function $\overline G_0\: \X \To \V_0$ that satisfies $J_{\V_0} \circ \overline G_0 = G_0$. The surjectivity of $\overline G_0$ is equivalent to the equation $\overline G_0 \circ \overline G_0^\dagger = I_{\V_0}$, which we apply in the following calculation:
\begin{align*}
\Eval \circ (J_{\V_0} \times I_\Y) \circ (I_{\V_0} \times S)
& =
\Eval \circ (J_{\V_0} \times I_\Y) \circ  (\overline G_0 \times I_\Y) \circ (\overline G_0^\dagger \times I_\Y) \circ  (I_{\V_0} \times S)
\\ &=
\Eval \circ  ( G_0 \times I_\Y)  \circ  (I_{\X} \times S)\circ (\overline G_0^\dagger \times I_\Y)
\\ &=
F \circ  (I_{\X} \times S)\circ (\overline G_0^\dagger \times I_\Y)
\\ & \leq
F \circ  (R \times S)\circ (\overline G_0^\dagger \times I_\Y)
\\ &\leq
T \circ F \circ (\overline G_0^\dagger \times I_\Y)
\\ &=
T \circ \Eval \circ  ( G_0 \times I_\Y) \circ (\overline G_0^\dagger \times I_\Y)
\\ & =
T \circ \Eval \circ (J_{\V_0} \times I_\Y) \circ  (\overline G_0 \times I_\Y) \circ (\overline G_0^\dagger \times I_\Y)
\\ &=
T \circ \Eval \circ (J_{\V_0} \times I_\Y)
\end{align*}
Thus, $\Eval \circ (J_{\V_0} \times I_\Y)$ is a homomorphism from $(\V_0 \times \Y, I_{\V_0} \times S)$ to $(\Z, T)$, i.e., $\V_0 \in \fF$. The maximality of $\W$ now implies that $\V_0$ is a subset of $\W$, so we can define $G = J_{\V_0}^\W \circ\overline G_0$. 

The function $G$ makes the diagram commute on the level of quantum sets and functions:
\begin{align*}
\Eval \circ (J_\W \times I_\Y) \circ & (G \times I_Y)
=
\Eval \circ ((J_\W \circ G) \times I_\Y)
=
\Eval \circ ((J_\W \circ J_{\V_0}^\W \circ \overline G_0) \times I_\Y)
\\ & =
\Eval \circ ((J_{\V_0} \circ \overline G_0) \times I_\Y)
=
\Eval \circ ( G_0 \times I_\Y)
=
F
\end{align*}
For any other function $G'\: \X \To \W$ making the diagram commute, we have $$\Eval \circ ((J_\W \circ G') \times I_\Y) = \Eval \circ ((J_\W \circ G) \times I_\Y).$$ The universal property of the evaluation function $\Eval\:\Z^\Y \times \Y \To \Z$ then implies that $J_\W \circ G' = J_\W \circ G$. The inclusion $J_\W$ is injective, so we conclude that $G' = G$. Therefore, $G$ is the unique function making the diagram commute.

To show that $G$ is a homomorphism, we first show that $\Eval \circ (J_\W \times I_\Y)$ is a homomorphism from $(\W \times \Y, (G \circ R \circ G^\dagger) \times S)$ to $(\Z, T)$:
\begin{align*}
\Eval \circ (J_\W \times I_\Y) \circ  ((G \circ R \circ G^\dagger) \times S)
& =
\Eval \circ (J_\W \times I_\Y) \circ (G \times I_\Y) \circ (R \times S) \circ (G^\dagger \times I_\Y)
 \\ &=
F \circ (R \times S) \circ (G^\dagger \times I_\Y)
\\ & \leq
T \circ F \circ (G^\dagger \times I_\Y)
\\ & =
T \circ \Eval \circ (J_\W \times I_\Y) \circ (G \times I_\Y) \circ (G^\dagger \times I_\Y)
\\ &=
T \circ \Eval \circ (J_\W \times I_\Y) \circ ((G \circ G^\dagger) \times I_\Y)
\\ & \leq
T \circ \Eval \circ (J_\W \times I_\Y)
\end{align*}
By definition of $Q$, we find that $G \circ R \circ G^\dagger \leq Q$. Composing on the right by $G$, we conclude that $G \circ R \leq G \circ R \circ G^\dagger \circ G \leq Q \circ G$. In other words, we conclude that $G$ is a homomorphism from $(\X, R)$ to $(\W,Q)$.

We have shown that for every homomorphism $F\:(\X \times \Y, R \times S) \To (\Z, T)$, there is indeed a unique homomorphism $G\: (\X, R) \To (\W,Q)$ making the diagram commute. This universal property determines the structure $(\W,Q)$ up to canonical isomorphism, in the usual way. To establish equality, let $(\W',Q')$ be another structure with $\W \subseteq \Z^\Y$ that satisfies properties (1) and (2). The canonical isomorphism $G\: (\W',Q') \To (\W,Q)$ makes the appropriate diagram commute, and therefore, it satisfies $\Eval \circ(J_{\W'} \times I_\Y) = \Eval \circ(J_{\W} \times I_\Y) \circ (G \times I_\Y)$. As before, we appeal to the universal property of $\Eval$ to infer that $J_{\W'} = J_\W \circ G$. By proposition 10.1 of \cite{Kornell2}, we conclude that $\W'$ is a subset of $\W$. Similarly, $\W$ is a subset of $\W'$, so the two quantum sets are equal. The equation $J_{\W'} = J_\W \circ G$ now gives us $J_{\W} \circ G = J_{\W'} = J_\W = J_\W \circ I_\W $. Appealing to the injectivity of $J_\W$, we conclude that $G = I_\W$. This bijection is a homomorphism in both directions, so we have the following chain of inequalities:
$$Q = Q \circ I_\W \leq I_\W \circ Q ' = Q' = Q' \circ I_\W \leq I_\W \circ Q = Q$$
Therefore, $W' = W$ and $Q' = Q$.
\end{proof}

\begin{theorem}\label{thm:qPOS is monoidal closed}
The category $\qPOS$ is monoidal closed with respect to the monoidal product $\times$, i.e., for each pair of quantum posets $(\Y,S)$ and $(\Z,T)$, there exist a quantum poset $([\Y,\Z]_{\sqsubseteq},Q)$ and a monotone function $\Eval_{\sqsubseteq}:[\Y,\Z]_{\sqsubseteq}\times \Y\to\Z$
such that for each quantum poset $(\X,R)$ and each monotone function $F:\X\times \Y\to \Z$,
there is a unique monotone function $G \: \X \To  [\Y,\Z]_{\sqsubseteq}$ satisfying $\Eval_{\sqsubseteq} \circ ( G \times I_\Y) = F$:
\[\begin{tikzcd}
(\X \times \Y, R \times S)
\arrow{rrrd}{F}
\arrow[dotted]{d}[swap]{G \times I_\Y}
&
&
&
\\
([\Y,\Z]_{\sqsubseteq} \times \Y, Q \times S)
\arrow{rrr}[swap]{\mathrm{Eval_{\sqsubseteq}}}
&
&
&
(\Z, T).
\end{tikzcd}\]
Furthermore, $[\Y,\Z]_{\sqsubseteq}$ is the largest subset of $\Z^\Y$ such that evaluation restricted to $[\Y,\Z]_{\sqsubseteq} \times \Y$ is a monotone function $([\Y,\Z]_{\sqsubseteq} \times \Y, I \times S) \to (\Z, T)$, and $Q$ is the largest binary relation on $[\X,\Y]_\sqsubseteq$ such that
\begin{equation}\label{ineq:Q def}
\Eval_{\sqsubseteq}\circ (Q\times S)\leq T\circ \Eval_{\sqsubseteq}.
\end{equation}
\end{theorem}

\begin{proof}
We define $[\Y,\Z]_\sqsubseteq$ to be the quantum set $\W$ in Lemma \ref{lem:monoidal closure lemma}, and $Q$ is obtained via the same lemma.
Furthermore, we define $\Eval_{\sqsubseteq}:=\Eval\circ (J_\W\times I_\Y)$.
By definition of $\Eval_\sqsubseteq$ and $Q$ it follows that $Q$ is the largest binary relation on $[\Y,\Z]_\sqsubseteq$ such that Equation (\ref{ineq:Q def}) holds. We need to show that $Q$ is an order on $\W$. It satisfies $I_\W\leq Q$ by Lemma \ref{lem:monoidal closure lemma}. To establish transitivity, we compute that
\begin{align*}
    \Eval_{\sqsubseteq}\circ ((Q\circ Q) \times S) & = \Eval_{\sqsubseteq}\circ ((Q\circ Q) \times (S\circ S))=\Eval_{\sqsubseteq}\circ (Q\times S)\circ (Q\times S) \\
    & \leq T\circ \Eval_{\sqsubseteq}\circ (Q\times S)  \leq T^2\circ \Eval_{\sqsubseteq} \leq T\circ \Eval_{\sqsubseteq},
\end{align*}
where we use Equation (2) of Lemma \ref{lem:basic facts on quantum orders} in the first equality. Since $Q$ is the largest binary relation on $\W$ satisfying Equation (\ref{ineq:Q def}), we obtain $Q\circ Q\leq Q$. Thus, we have proven that $Q$ is a pre-order. 

Assume that $Q$ is not an order, so $Q\wedge Q^\dag\neq I_\W$. Let $E=Q\wedge Q^\dag$. Then $E$ is an equivalence relation by Lemma \ref{lem:pre-order generates equivalence relation}. Since $Q$ is not an order, we have $E\neq I_\W$; hence Lemma \ref{lem:equivalence relation lemma} yields a function $K:\W\to\W$ 
such that $K\neq I_\W$ and $K\leq E$.

We have that $Q\times S\leq \Eval_{\sqsubseteq}^\dag\circ \Eval_{\sqsubseteq}\circ (Q\times S)\leq \Eval_{\sqsubseteq}^\dag\circ T\circ \Eval_{\sqsubseteq}$
 by definition of a function and Equation (\ref{ineq:Q def}), whence
$Q\times I_\Y\leq \Eval_{\sqsubseteq}^\dag\circ T\circ \Eval_{\sqsubseteq}.$
Applying the adjoint operation to both sides of this inequality, we obtain 
$Q^\dag\times I_\Y \leq \Eval_{\sqsubseteq}^\dag\circ T^\dag\circ \Eval_{\sqsubseteq}.$
Thus,
\begin{align*}
K\times I_\Y & \leq E\times I_\Y  = (Q\wedge Q^\dag) \times I_\Y= (Q\times I_\Y)\wedge (Q^\dag\times I_\Y) \\
    & \leq (\Eval_{\sqsubseteq}^\dag\circ T\circ \Eval_{\sqsubseteq})\wedge (\Eval_{\sqsubseteq}^\dag\circ T^\dag\circ \Eval_{\sqsubseteq})\\
    & = \Eval_{\sqsubseteq}^\dag\circ (T\wedge T^\dag)\circ \Eval_{\sqsubseteq}= \Eval_{\sqsubseteq}^\dag\circ \Eval_{\sqsubseteq},
\end{align*}
where we appeal to Lemma \ref{lem:product distributes over joins} in the second equality and to Proposition \ref{prop:intersection} in the penultimate equality.
Hence,
$\Eval_{\sqsubseteq}\circ(K\times I_\Y) \leq  \Eval_{\sqsubseteq}\circ \Eval_{\sqsubseteq}^\dag\circ \Eval_{\sqsubseteq}\leq \Eval_{\sqsubseteq},$
and by Lemma \ref{lem:inequality between functions is equality} it follows that 
$\Eval_{\sqsubseteq}\circ(K\times I_\Y) = \Eval_{\sqsubseteq},$
or equivalently,
$\Eval\circ (J_\W\times I_\Y)\circ (E\times I_\Y)=\Eval\circ (J_\W\times I_\Y)\circ (I_\W\times I_\Y).$
By the universal property in Lemma \ref{lem:monoidal closure lemma}, it follows that $K\times I_\Y=I_\W\times I_\Y$. By Lemma \ref{lem:projection maps are surjective} and Proposition \ref{prop:tensor product of functions}, we obtain
\[  K=K\circ P\circ P^\dag = P\circ (K\times I_\Y)\circ P^\dag = P\circ (I_\W\times I_\Y)\circ P^\dag = I_\W\circ P\circ P^\dag = I_\W,\]
contradicting our choice of $K$. We conclude that $Q$ must be an order. It then follows immediately from Equation (\ref{ineq:Q def}) that $\Eval_{\sqsubseteq}$ is monotone. The claimed universal property of $([\Y,\Z]_\below, \Eval_\below)$ is just the universal property of Lemma \ref{lem:monoidal closure lemma} because a monotone function is just a homomorphism between two quantum sets equipped with preorders.
\end{proof}

If $\Y$ is trivially ordered, i.e., if $S = I_\Y$, then $[\Y,\Z]_\below$ is equal to $\Y^\Z$ as a quantum set because $\Eval \circ (J_\W \times I_\Y)$ is trivially monotone as a function $\W \times \Y \to \Z$. Hence, the quantum function set $\Y^\Z$ is canonically ordered.

\begin{example}
Let $\X$ be a quantum set, and as before, let $\mathbf 1$ be the quantum set whose only atom is $\CC$. Then, $[\X, `\RR]_\below$ is equal to $`\RR^\X$ as a quantum set, and hence the quantum function set $`\RR^\X$ is canonically ordered. Functions $\mathbf 1 \to `\RR^\X$ are in canonical bijection with functions $\X \to `\RR$, which are in canonical bijection with the self-adjoint operators in $\ell(\X)$, the set of all operators that are affiliated with $\ell^\infty(\X)$ \cite{Kornell2}*{Proposition 11.2}. In this sense, $`\RR^\X$ is the quantum set of observables on $\X$. Under this correspondence, the order on functions $\mathbf 1 \to `\RR^\X$ corresponds to the so-called spectral order on self-adjoint operators \cite{Olson}: We reason in the notation of \cite{Kornell2}*{Proposition 11.2}, where $e_\alpha \in \ell^\infty(`\RR)$ is the atomic projection corresponding to a real number $\alpha$ and $r \in \ell(`\RR)$ is the self-adjoint operator defined by $re_\alpha = \alpha e_\alpha$ for all $\alpha \in \RR$. Using brackets to notate spectral projections and real numbers to denote the corresponding functions $\{*\} \to \RR$, we reason that for all functions $F, G\: \X \to `\RR$,
\begin{align*}
F \below G
& \EV
G \circ F^\dagger \leq `(\leq)
\\ & \EV
\text{for all }\alpha, \beta \in \RR, \quad `\beta^\dagger \circ G \circ F^\dagger \circ `\alpha \leq `(\beta^\dagger \circ (\leq) \circ \alpha)
\\ &\EV
\text{for all }\alpha, \beta \in \RR, \quad \text{if } `\alpha^\dagger \circ F \not \perp `\beta^\dagger \circ G, \quad \text{then } \alpha \leq\beta
\\ & \EV
\text{for all } \alpha > \beta \in \RR, \quad  `\alpha^\dagger \circ F \perp `\beta^\dagger \circ G
\\ & \EV
\text{for all } \alpha > \beta \in \RR, \quad F^\star(e_\alpha) \perp G^\star(e_\beta)
\\ & \EV
\text{for all } \alpha > \beta \in \RR, \quad [F^\star(r) = \alpha] \perp [G^\star(r) = \beta]
\\ & \EV
\text{for all } \lambda \in \RR, \quad [F^\star(r) \leq \lambda] \geq [G^\star(r) \leq \lambda].
\end{align*}
The first equivalence follows by Lemma \ref{lem:order between functions}; the second equivalence follows by Lemma \ref{lem:coproduct order}(b); the third equivalence follows by Proposition \ref{appendix.A.2}; the fourth equivalence is logical; the fifth equivalence follows by \cite{Kornell2}*{Proposition B.8}; the sixth equivalence follows by the functional calculus. The seventh equivalence follows from \cite{Kornell2}*{Proposition 5.4} because this proposition implies that $[a \leq \lambda] = \sum_{\alpha \leq \lambda} [a = \alpha]$ for all $a \in \ell^\infty(\X)$ and all $\lambda \in \RR$ and hence for all $a \in \ell(\X)$ and all $\lambda \in \RR$.
\end{example}

\section{Quantum power set}\label{sec:embedding into the powerset}

The power set functor $\Pow\: \cat{Rel} \to \cat{Set}$ is defined to map each set $A$ to its power set $\Pow(A) := \{C : C \subsetof A \}$ and each binary relation $r \in \cat{Rel}(A,B)$ to its direct image function $\Pow(r)\: C \mapsto \{b \in B : (a,b) \in r\text{ for some } a \in C\}$. Up to natural isomorphism, this power set functor may also be defined as the right adjoint of the inclusion functor $\Inc\: \cat{Set} \to \cat{Rel}$ \cite{JencovaJenca}. We show that the inclusion functor $\Inc\: \cat{qSet} \to \cat{qRel}$ also has a right adjoint, thereby obtaining a quantum analog of the power set functor.

For each ordinary set $A$, the power set $\Pow(A)$ is canonically isomorphic to the function set $\BB^A$, where $\BB := \{0,1\}$; this motivates our definition of the quantum power set:

\begin{definition}\label{def:power set}
Let $\X$ be a quantum set. We define the \emph{quantum power set} of $\X$ to be the quantum set $\qPow(\X):=(`\BB)^{\X^*}$. The notation $\X^*$ refers to the dual of the quantum set $\X$, which is obtained by dualizing each atom of $\X$ \cite{Kornell2}*{Definition 3.4}. For each ordinary set $A$, the quantum set $`A$ is naturally isomorphic to its dual $(`A)^*$, so this feature of the definition is only significant when $\X$ has an atom of dimension larger than $1$. The quantum analogue of the membership relation is the adjoint of the unique binary relation $\ni_\X$ from $\Pow(\X)$ to $\X$ such that the following diagram in $\mathbf{qRel}$ commutes:
$$
\begin{tikzcd}
\qPow(\X) \times \X^* 
\arrow{r}{\Eval}
\arrow[dotted]{d}[swap]{\ni_\X \times I_{\X^*}}
&
`\BB
\arrow{d}{`1^\dagger}
\\
\X \times \X^*
\arrow{r}[swap]{E_\X}
&
\mathbf 1.
\end{tikzcd}
$$
Here, $1\: \{\ast\} \to \BB$ is the function corresponding to $1 \in \BB$, and $E_\X\: \X \times \X^* \to \mathbf 1$ is the counit of the duality between $\X$ and $\X^*$ \cite{Kornell2}*{section III}. Both the existence and the uniqueness of the binary relation $\ni_\X$ follow immediately from the fact that $\X$ together with $E_\X$ is an inner hom from $\X^*$ to $\mathbf 1$ in the symmetric monoidal category $\cat{qRel}$ \cite{Kornell2}*{Theorem 3.6}.
\end{definition}

\begin{theorem}\label{thm:power set}
The assignment $\X\mapsto\qPow(\X)$ extends to a functor $\qPow:\qRel\to\qSet$ that is right adjoint to the inclusion $\Inc:\qSet\to\qRel$. The binary relation $\ni_\X$ is the $\X$-component of the counit for each quantum set $\X$.
\end{theorem}
\begin{proof}
Let $\Y$ be a quantum set, and let $R$ be a binary relation from $\Y$ to $\X$. We show that there exists a unique function $F_R:\Y\to\qPow(\X)$ such that the following two equivalent diagrams commute:
\[\begin{tikzcd}
\Inc(\Y)\ar{rd}{R}\ar{d}[swap]{\Inc(F_R)} &     & & \Y\ar{rd}{R}\ar{d}[swap]{F_R}     \\
(\Inc\circ\qPow)(\X)\ar{r}[swap]{\ni_\X} & \X, & &  `\BB^{\X^*}\ar{r}[swap]{\ni_\X} & \X.
\end{tikzcd}\]
First, we observe that $\X$ together with $E_\X$ is an inner hom from $\X^*$ to $\mathbf 1$ in the dagger compact category $\cat{qRel}$, so we have a natural bijection
\begin{equation*}\label{eq:power set C}
\qRel(\Y,\X)\to\qRel(\Y\times\X^*,\mathbf 1), \qquad S\mapsto E_{\X}\circ (S\times I_{\X^*}),
\end{equation*} which
maps $R$ to $E_{\X}\circ (R\times I_{\X^*})$. Next, by \cite{Kornell2}*{Theorem B.8}, we have a natural bijection
\[\qSet(\Y\times\X^*,`\BB)\to\qRel(\Y\times\X^*,\mathbf 1),\qquad F\mapsto `1^\dag\circ F;\]
hence there is a unique function $G_R:\Y\times\X^*\to`\BB$ such that 
$`1^\dag\circ G_R=E_{\X}\circ (R\times I_{\X^*}).$
Finally, since $\qSet$ is monoidal closed, we have a natural bijection
\[\qSet(\Y,\mathbf `\BB^{\X^*})\to\qSet(\Y\times\X^*,`\BB), \qquad F\mapsto\Eval\circ (F\times I_{\X^*});\]
hence there is a unique function $F_R:\Y\to\mathbf `\BB^{\X^*}$ such that $G_R=\Eval\circ (F_R\times I_{\X^*})$.
Thus, there is a function $F_R:\Y\to\mathbf `\BB^{\X^*}$ such that
\begin{equation*}\label{eq:power set B}
    `1^\dag\circ \Eval\circ (F_R\times I_{\X^*})=E_{\X}\circ (R\times I_{\X^*}).
\end{equation*}
Therefore, by the definition of $\ni_\X$,
\[E_{\X}\circ ( (\ni_\X\circ F_R)\times I_\X)=E_{\X}\circ (\ni_\X\times I_{\X^*})\circ (F_R\times I_{\X^*})=`1^\dag\circ\Eval\circ (F_R\times I_{\X^*})=E_{\X}\circ (R\times I_\X).\]
Appealing again to the fact that $\X$ together with $E_\X$ is the inner hom from $\X^*$ to $\mathbf 1$ in the dagger compact category $\cat{qRel}$, we conclude that $\ni_\X\circ F_R=R$.

To establish the uniqueness of $F_R$, let $F:\Y\to `\BB^{\X^*}$ be a function such that $\ni_\X\circ F=R$. Then $\ni_\X\circ F=\ni_\X\circ F_R$, and hence 
\begin{align*}
    `1^\dag\circ\Eval\circ (F\times I_{\X^*}) & = E_{\X}\circ (\ni_\X\times I_{\X^*})\circ (F\times I_{\X^*})= E_{\X}\circ ((\ni_\X\circ F)\times I_{\X^*})\\
& = E_{\X}\circ ((\ni_\X\circ F_R)\times I_{\X^*})= E_{\X}\circ (\ni_\X\times I_{\X^*})\circ (F_R\times I_{\X^*})\\
& =`1^\dag\circ\Eval\circ (F_R\times I_{\X^*}),
\end{align*}
where we use the diagram in the statement defining $\ni_\X$ in the first and the last equalities. Appealing again to \cite{Kornell2}*{Theorem B.8}, we find that $\Eval\circ (F\times I_{\X^*})=\Eval\circ (F_R\times I_{\X^*}).$ Because $`\BB^{\X^*}$ together with $\Eval\: `\BB^{\X^*} \times \X^* \to `\BB$ is the inner hom from $\X^*$ to $`\BB$ in $\cat{qSet}$, we conclude that $F=F_R$.

Altogether, we have shown that for each quantum set $\Y$ and each binary relation $R$ from $\Y$ to $\X$, there exists a unique function $F_R$ from $\Y$ to $`\BB^{X^*}$ such that $\ni_\X \circ F_R = R$, or in other words, there exists a unique function $F_R$ from $\Y$ to $\qPow(\X)$ such that $\ni_X \circ \Inc (F_R) = R$. It follows that the functor $\Inc$ has a right adjoint which maps each quantum set $\X$ to $\qPow(\X)$ \cite[Theorem IV.1.2.(iv)]{MacLane}, as claimed.
\end{proof}

Each ordinary poset $(A, {\below})$ may be embedded into the poset $(\Pow(A), {\subsetof})$ by mapping each element $a \in A$ to its principal down set $\downarrow\!a : = \{a' \in A : a' \below a\}$. We generalize this proposition to the quantum setting.

Fix a quantum set $\X$. The quantum power set $\qPow(\X)$ is canonically ordered: Let $\X^*$ be ordered flatly, that is by $I_{\X^*}$. Let $\BB$ be ordered such that $0 \sqsubset 1$. Applying Theorem \ref{thm:qPOS is monoidal closed}, we obtain a quantum poset $[\X^*, `\BB]_\below$ whose underlying quantum set $\W$ is $`\BB^{\X^*}$, i.e., $\qPow(\X)$. Indeed, Lemma \ref{lem:monoidal closure lemma} characterizes $\W$ as the largest subset of $`\BB^{\X^*}$ such that $\Eval\circ (J_\W\times I_{\X^*}):(\W\times\X^*,I_\W\times I_{\X^*})\to (`\BB,`{\below})$ is a homomorphism, and it follows from Example \ref{ex:function with trivially ordered domain is monotone} that $\Eval\circ (J_\W\times I_{\X^*})$ is a homomorphism for any subset $\W$ of $`\BB^{\X^*}$.

Fix an order $R$ on $\X$. Then, $P : = E_{\X} \circ (R^\dagger \times I_{\X^*})$ is a binary relation from $\X \times \X^*$ to $\mathbf 1$, and there is a unique function $\tilde P\: \X \times \X^* \to \mathbf `\BB$ such that $`1^\dagger \circ \tilde P  = P$ \cite{Kornell2}*{Theorem B.8}. Intuitively, the function $\tilde P$ yields $1$ if and only if its first argument is above its second argument. Appealing to the universal property of the evaluation function, we obtain a function $G\: \X \to \qPow(\X)$ such that the following diagram in $\cat{qSet}$ commutes:
$$
\begin{tikzcd}
\X \times \X^*
\arrow{d}[swap]{G \times I_{\X^*}}
\arrow{rd}{\tilde P}
\\
\qPow(\X) \times \X^* 
\arrow{r}[swap]{\Eval}
&
`\BB.
\end{tikzcd}
$$
\noindent We will show that this function $G$ is an order embedding $(\X,R) \to [(\X^*,I_{\X^*}),(`\BB,`{\below})]_{\below}$.

\begin{proposition}\label{prop:upsets are monotone}
Let $(\Y, S)$ be a quantum poset, and let $Q$ be a binary relation from $\Y$ to $\mathbf 1$. Then, $Q \circ S^\dagger = Q$ if and only if the unique function $\tilde Q\: \Y \to `\BB$ such that $`1^\dagger \circ \tilde Q = Q$ is monotone.
\end{proposition}

\begin{proof}
We reason in terms of the trace on $\cat{qRel}$ (Appendix \ref{appendix.A}), as follows:
\begin{align*}
& \text{$\tilde Q$ is monotone}
 \EV
\tilde Q \circ S \circ \tilde Q^\dagger \leq `({\below})
\EV
\tilde Q \circ S \circ \tilde Q^\dagger \perp `(\not \sqsubseteq) = `0 \circ `1^\dagger
\\ & \EV
\Tr((\tilde Q \circ S \circ \tilde Q^\dagger)^\dagger \circ `0 \circ `1^\dagger ) = \bot
\EV
\Tr(`1^\dagger \circ \tilde Q \circ S^\dagger \circ \tilde Q^\dagger \circ `0) = \bot
\\ & \EV
\Tr(Q \circ S^\dagger \circ \neg Q^\dagger) = \bot
\EV
Q \circ S^\dagger \perp \neg Q
\EV
Q \circ S^\dagger \leq  Q
\\ & \EV
Q \circ S^\dagger = Q.
\end{align*}
In this computation, we use the fact that $\tilde Q^\dagger \circ `0 = (`0^\dagger \circ \tilde Q)^\dagger = \neg (\neg (`0^\dagger \circ \tilde Q))^\dagger = \neg (`1^\dagger \circ \tilde Q)^\dagger = \neg Q^\dag$. The equality $\neg (`0^\dagger \circ \tilde Q )= `1^\dagger \circ \tilde Q$ holds because binary relations from $`\BB$ to $\mathbf 1$ correspond to projections in $\ell^\infty(`\BB)$ and precomposition by functions corresponds to the application of unital normal $*$-homomorphisms \cite{Kornell2}*{Theorem B.8}.
\end{proof}

\begin{lemma}\label{lem:tilde P is monotone}
The function $\tilde P$ is monotone $(\X \times \X^*, R \times I_{\X^*}) \to (`\BB, `{\below})$. Furthermore, for all orders $T$ on $\X$, if $\tilde P$ is monotone $(\X \times \X^*, T \times I_{\X^*}) \to (`\BB, `{\below})$, then $T \leq R$.
\end{lemma}

\begin{proof}
We calculate that 
$P \circ (R \times I_{\X^*})^\dagger  = E_{\X} \circ (R^\dagger \times I_{\X^*}) \circ (R^\dagger \times I_{\X^*}) = E_{\X}  \circ (R^\dagger \times I_{\X^*})  = P$. Therefore, 
by Proposition \ref{prop:upsets are monotone}, $\tilde P$ is a monotone function $(\X \times \X^*, R \times I_{\X^*}) \to (`\BB, `{\below})$. Now, let $T$ be an order on $\X$, and assume that $\tilde P$ is a monotone function $(\X \times \X^*, T \times I_{\X^*}) \to (`\BB, `{\below})$. By Proposition \ref{prop:upsets are monotone}, we find that $P \circ (T \times I_{\X^*})^\dagger = P$. We now calculate that
\begin{align*}
E_{\X} \circ (R^\dagger \times I_{\X^*})
& =
P
=
P \circ (T \times I_{\X^*})^\dagger
=
E_\X \circ (R^\dagger \times I_{\X^*}) \circ (T \times I)^\dagger
=
E_\X \circ ((T \circ R)^\dagger \times I_{\X^*}).
\end{align*}
We conclude that $T \circ R = R$. Therefore $T = T \circ I_{\X} \leq T \circ R = R$, as claimed.
\end{proof}

\begin{theorem}\label{thm:embedding into the powerset}
Let $\X$ be a quantum set equipped with an order $R$. Let $\tilde P\: \X \times \X^* \to `\BB$ be the function defined by $`1^\dagger \circ \tilde P = E_{\X} \circ (R^\dagger \times I_{\X^*})$. Equip $\X^*$ with the trivial order $I_{\X^*}$. The unique monotone function $G$ that makes the following diagram in $\cat{qPOS}$ commute is an order embedding:
$$
\begin{tikzcd}
\X \times \X^*
\arrow{d}[swap]{G \times I_{\X^*}}
\arrow{rd}{\tilde P}
\\
{}[\X^*, `\BB]_\below \times \X^* 
\arrow{r}[swap]{\Eval_\below}
&
`\BB.
\end{tikzcd}
$$
\end{theorem}

\begin{proof}
By Lemma \ref{lem:tilde P is monotone}, $\tilde P$ is monotone, so Theorem \ref{thm:qPOS is monoidal closed} guarantees the existence of such a monotone function $G$. We claim that $G$ is injective. Let $\W$ be a quantum set, and let $F_1$ and $F_2$ be functions $\W \to \X$. Assume that $G \circ F_1 = G \circ F_2$. We now reason as follows:
\begin{align*}
G \circ F_1 = G \circ F_2
& \IMP
`1^\dagger \circ \Eval_\below \circ (G \times I_{\X^*}) \circ (F_1 \times I_{\X^*}) = `1^\dagger \circ \Eval_\below \circ (G \times I_{\X^*}) \circ (F_2 \times I_{\X^*})
\\ &
\EV `1^\dag\circ \tilde P\circ (F_1\times I_{\X^*})=`1^\dag\circ \tilde P\circ (F_2\times I_{\X^*}) 
\\ & \EV
E_\X \circ (R^\dagger \times I_{\X^*}) \circ (F_1 \times I_{\X^*}) = E_{\X} \circ (R^\dagger \times I_{\X^*}) \circ (F_2\times I_{\X^*})
\\ & \EV
R^\dagger \circ F_1 = R^\dagger \circ F_2
\EV
F_1 \below F_2 \text{ and } F_2 \below F_1
\\ & \EV F_1 = F_2.
\end{align*}
We conclude that $G$ is monic in $\cat{qSet}$, and it is therefore injective \cite{Kornell2}*{Proposition 8.4}.

Let $Q$ be the binary relation that orders $[\X^*, `\BB]_\below$.
Since $G$ is injective, the binary relation $T = G^\dagger \circ Q\circ G$ is an order on $\X$ (Lemma \ref{lem:pullback order}). The function $G$ is then an order embedding $(\X, T) \to ([\X^*, `\BB]_\below, Q)$, and in particular, it is monotone. Hence, $\tilde P$ is a monotone function $(\X \times \X^*, T \times I_{\X^*}) \to (`\BB, `{\below})$. By Lemma \ref{lem:tilde P is monotone}, $T \leq R$. In other words, $G^\dagger \circ Q \circ G \leq R$. However, since $G$ is monotone, we also have $G^\dagger \circ Q\circ G \geq R$. Therefore, $G^\dagger \circ Q \circ G = R$; in other words, $G$ is an order embedding $(\X, R) \to ([\X^*, `\BB]_\below, Q)$.
\end{proof}

The quantum poset $[\X^*, `\BB]_\below$ is nothing but $\qPow(\X)$ equipped with its canonical order. Thus, $G$ is an order embedding $\X \to \qPow(\X)$.

\appendix

\section{Quantum sets} We record a number of basic facts about quantum sets and the binary relations between them in the sense of \cite{Kornell2}, which serves as our basic reference.

\begin{lemma}\label{lem:union}\label{lem:intersection}
Let $\W$, $\X$, $\Y$ and $\Z$ be quantum sets, let $R\in\qRel(\W,\X)$, let $\{S_i\}_{i\in I}\subseteq\qRel(\X,\Y)$ and let $T \in\qRel(\Y,\Z)$. Then,
\[ \bigvee_{i\in I}(S_i\circ R)=\left(\bigvee_{i\in I}S_i\right)\circ R, \qquad \bigvee_{i\in I}(T\circ S_i)=T\circ \left(\bigvee_{i\in I}S_i\right),\]
\[ \bigwedge_{i\in I}(S_i\circ R)\geq\left(\bigwedge_{i\in I}S_i\right)\circ R,\qquad  \bigwedge_{i\in I}(T\circ S_i)\geq T\circ \left(\bigwedge_{i\in I}S_i\right).\]
\end{lemma}

\begin{proof}
We prove the last formula. For all atoms $X \atomof \X$ and $Z \atomof \Z$, we calculate that
\begin{align*}
&\left(\bigwedge_{i\in I}(T\circ S_i)\right)(X,Z)
 =
\bigwedge_{i\in I} (T\circ S_i)(X,Z)
 = 
\bigwedge_{i\in I} \bigvee_{Y \atomof \Y} (T(Y,Z)\cdot S_i(X,Y))
\\ & \geq 
\bigvee_{Y \atomof \Y} \bigwedge_{i\in I} (T(Y,Z)\cdot S_i(X,Y))
 \geq
\bigvee_{Y \atomof \Y}   \left (T(Y,Z)  \cdot  \bigwedge_{i\in I}S_i(X,Y)\right)
 =
\left( T \circ \bigwedge_{i\in I}S_i\right)(X,Z).
\end{align*}
The other three formulas are proved similarly.
\end{proof}

\begin{lemma}\label{lem:action of dagger}
Let $\X_1$ and $\X_2$ be quantum sets, and let $R$, $S$, $\{T_i\}_{i\in I}$ be binary relations from $\X_1$ to $\X_2$. Then,
\begin{itemize}
    \item[(1)] $R\leq S$ if and only if $R^\dag\leq S^\dag$;
    \item[(2)] $\left(\bigwedge_{i\in I}T_i\right)^\dag=\bigwedge T_i^\dag$;
    \item[(3)] $\left(\bigvee_{i\in I}T_i\right)^\dag=\bigvee T_i^\dag$;
    \item[(4)] $(\neg S)^{\dag}=\neg (S^{\dag})$.
\end{itemize}
\end{lemma}

\begin{proof}
In each case, the proof proceeds by fixing arbitrary atoms $X_1 \atomof \X_1$  and $X_2 \atomof \X_2$ and then verifying the equivalence or equality in question in the $(X_1, X_2)$-component, appealing to the same equivalence or equality for arbitrary subspaces of $L(X_1, X_2)$.
\end{proof}

\begin{lemma}\label{lem:product distributes over joins}
Let $\X_1$, $\X_2$, $\Y_1$ and $\Y_2$ be quantum sets. Let $R\in\qRel(\X_1,\Y_1)$, and $S\in\qRel(\X_2,\Y_2)$. Let $\{R_\alpha\}_{\alpha\in A}$ be an indexed family in $\qRel(\X_1,\Y_1)$, and let $\{S_\beta\}_{\beta\in B}$ be an indexed family in $\qRel(\X_2,\Y_2)$. Then,
\begin{itemize}
    \item[(a)] $(R\times S)^\dag=R^\dag\times S^\dag$;
    \item[(b)] $\left(\bigwedge_{\alpha\in A}R_\alpha\right)\times \left(\bigwedge_{\beta\in B}S_\beta\right)=\bigwedge_{\alpha\in A}\bigwedge_{\beta\in B}(R_\alpha\times S_\beta)$;
    \item[(c)] $\left(\bigvee_{\alpha\in A}R_\alpha\right)\times S=\bigvee_{\alpha\in A}(R_\alpha\times S)$. 
\end{itemize}
\end{lemma} 
\begin{proof}
In each case, the proof proceeds by fixing arbitrary atoms $X_1 \atomof \X_1$, $X_2 \atomof \X_2$, $Y_1 \atomof \Y_1$ and $Y_2 \atomof \Y_2$ and then verifying the equality in question in the $(X_1 \tensor X_2, Y_1 \tensor Y_2)$-component, appealing to the same equality for arbitrary subspaces of $L(X_1 , Y_1)$ and $L(X_2, Y_2)$.
\end{proof}

The next lemma concerns inclusion functions \cite{Kornell2}*{Definition 8.2}. For each atom $X$ of a quantum set $\X$ we abbreviate $J_X := J^\X_{\Q\{X\}}$ \cite{Kornell2}*{Definition 2.3}. The function $J_X\: \Q\{X\} \to \X$ is defined by $J_X(X,X) = \CC \cdot 1_X$, with all other components vanishing.

\begin{lemma}
Let $\X$ be a quantum set. Then $I_\X=\bigvee_{X\atomof\X}J_X\circ J_X^\dag$.
\end{lemma}
\begin{proof}
It is sufficient to observe that for all $X \atomof \X$, we have that $(J_X \circ J_X^\dagger)(X,X) = \CC \cdot 1_X$, with all the other components of $J_X \circ J_X^\dagger$ vanishing.
\end{proof}

The next lemma refers subsets of quantum sets \cite{Kornell2}*{Definition 2.2(3)}. A quantum set $\X_0$ is said to be subset of a quantum set $\X$ if each atom of $\X_0$ is also an atom of $\X$.

\begin{lemma}\label{lem:restriction and corestriction}
Let $\X$ and $\Y$ be quantum sets, let $\X_0 \subsetof \X$, let $\Y_0 \subsetof \Y$, and let $R\in\qRel(\X,\Y)$. Then, $(R \circ J_{X_0})(X,Y) = R(X,Y)$ for all atoms $X \atomof \X_0$ and $Y \atomof \Y$, and $(J_{Y_0} \circ R)(X,Y) = R(X,Y)$ for all atoms $X \atomof \X$ and $Y \atomof \Y_0$.
\end{lemma}
\begin{proof}
Both equalities follow easily by direct computation.
\end{proof}

\begin{proposition}\label{prop:intersection}
Let $\W$, $\X$, $\Y$ and $\Z$ be quantum sets, and let $\{S_i\}_{i\in I}\subseteq\qRel(\X,\Y)$. Then, for all functions $F\: \W \to \X$ and $G\: \Z \to \Y$, we have
$\bigwedge_{i\in I}(S_i\circ F)=\left(\bigwedge_{i\in I}S_i\right)\circ F$
and 
$\bigwedge_{i\in I}(G^\dag\circ S_i)=G^\dag\circ\left(\bigwedge_{i\in I}S_i\right).$
\end{proposition}
\begin{proof}
The inequality $\bigwedge_{i\in I}(S_i\circ F) \geq \left(\bigwedge_{i\in I}S_i\right)\circ F$ follows from Lemma \ref{lem:intersection}, as does the inequality $\bigwedge_{i\in I}(S_i\circ F \circ F^\dagger) \geq \left(\bigwedge_{i\in I}(S_i\circ F\right))\circ F^\dagger$. We apply the latter inequality in the following calculation:
\[   \bigwedge_{i\in I}(S_i\circ F)\leq\left(\bigwedge_{i\in I}(S_i\circ F)\right)\circ F^\dag\circ F\leq\left(\bigwedge_{i\in I}(S_i\circ F\circ F^\dag)\right)\circ F\leq\left(\bigwedge_{i\in I}S_i\right)\circ F.\]
Thus, we establish the first equality of the proposition. The second equality can be obtained from the first by taking its adjoint and replacing $S_i$ and $F$ by $S_i^\dag$ and $G$, respectively.
\end{proof}

\begin{lemma}\label{lem:inequality between functions is equality}
Let $\X$ and $\Y$ be quantum sets, and let $F$ and $G$ be functions $\X\to\Y$. If $F\leq G$, then $F=G$.
\end{lemma}
\begin{proof}
Since we also have that $F^\dag\leq G^\dag$, we find that
$G=G\circ I_\X \leq G\circ F^\dag\circ F\leq G\circ G^\dag\circ F\leq I_\Y \circ F=F,$
whence $F=G$.
\end{proof}

\section{Projection functions}\label{sub:projection functions} The monoidal product of quantum sets generalizes the ordinary Cartesian product in the sense that we have a natural isomorphism $`S \times `T \iso `(S \times T)$ for all sets $S$ and $T$. Furthermore, for all quantum sets $\X$ and $\Y$, we have projection functions $P_{\X \times \Y}^\X\: \X \times \Y \to \X$ and $P_{\X \times \Y}^\Y \: \X \times \Y \to \Y$ \cite{Kornell2}*{Section 10}. Explicitly, $P_{\X \times \Y}^\X$ and $P_{\X \times \Y}^\Y$ are defined by the equations $P_{\X \times \Y}^\X(X\otimes Y,X)  = \rho_X(\CC 1_X\otimes L(Y, \CC))$ and $P_{\X \times \Y}^\Y(X\otimes Y,Y)  =  \lambda_Y(L(X, \CC)\otimes\CC 1_Y)$, for $X \atomof \X$ and $\Y \atomof \Y$, with the other components vanishing, where where $\lambda$ and $\rho$ denote the left and right unitors in $\FdHilb$. For brevity, we will sometimes write $P^\X = P_{\X \times \Y}^\X$ and $P^\Y = P_{\X \times \Y}^\Y$.

\begin{lemma}\label{lem:projection maps are surjective}
Let $\X$ and $\Y$ be nonempty quantum sets, and let $P^\X:\X\times \Y\to\X$ and $P^\Y:\X\times\Y\to\Y$ be the two projection functions. Then, $P^\X$ and $P^\Y$ are surjective.
\end{lemma}
\begin{proof}
For all atoms $X \atomof \X$ and $Y \atomof \Y$, we have that 
$
 (P^\X \circ P^{\X\dagger})(X,X) \geq P^\X(X \tensor Y, X) \cdot P^\X(X \tensor Y, X)^\dagger  = \rho_X (\CC 1_\X \tensor L(Y, \CC)) (\CC 1_\X \tensor L(Y, \CC)^\dagger) \rho_X^\dagger = \rho_\X (\CC 1_\X \tensor \CC) \rho_\X^\dagger = \CC 1_X = I_\X(X,X).
$
Hence, $P \circ P^\dagger \geq I_\X$, that is, $P$ is surjective. Similarly, $Q$ is surjective.
\end{proof}

Let $\W$ be a quantum set, and let $F\: \W \to \X$ and $G\: \W \to \Y$ be functions. If there exists a function $(F,G)\: \W \to \X \times \Y$ such that $P^\X \circ (F, G) = F$ and $P^\Y \circ (F, G) = G$, then it is clearly unique (Lemma \ref{lem:projection maps are surjective}) \cite{Kornell2}*{Theorem 7.4, Proposition 8.1}. This justifies the notation $(F,G)$. However, such a function $(F,G)$ need not exist; if it does, we say that $F$ and $G$ are compatible \cite{Kornell2}*{Definition 10.3}. For this and other reasons, this generalization of the Cartesian product to quantum sets is at once conceptually natural and technically challenging. In this subsection, we resolve a few basic questions about it.

\begin{proposition}\label{prop:tensor product of functions}
Let $\V$, $\W$, $\X$ and $\Y$ be quantum sets, and let $F:\V\to\X$ and $G:\W\to\Y$ be functions. Then $F\times G:\V\times \W \to \X\times\Y$ is the unique function such that the following diagram commutes:
\[
\begin{tikzcd}
\V\arrow{d}[swap]{F} &
\V\times\W \arrow{l}[swap]{P^\V} \arrow{r}{P^\W}
\arrow{d}{F\times G}
& \W\arrow{d}{G}
\\
\X
&
\X \times \Y \arrow{l}{P^\X} \arrow{r}[swap]{P^\Y}
&
\Y.
\end{tikzcd}\]
In other words, $F \circ P^\V$ and $F \circ P^\V$ are compatible, and $F \times G = (F \circ P^\V, F \circ P^\W)$.
\end{proposition}
\begin{proof}This follows by direct calculation, using the naturality of the unitors $\rho$ and $\lambda$.
\end{proof}

\begin{proposition}\label{prop:JXtimesJYisJXtimesY}
Let $\X$ and $\Y$ be quantum sets, and let $\V\subseteq\X$ and $\W\subseteq\Y$. Then $J_{\V\times\W}=J_\V\times J_\W$.
\end{proposition}

\begin{proof}
For all $V \atomof \V$ and $W \atomof \W$, we calculate that $J_{\V \times \W}(V\tensor W, V \tensor W) = \CC 1_{V \tensor W} = \CC1_V \tensor \CC1_W = J_\V(V,V) \tensor J_\W(W,W) = (J_\V \times J_\W)(V \tensor W, V \tensor W)$. Reasoning similarly, we may show that the other components of both $J_{ \V \times \W}$ and $J_\V \times J_\W$ are zero.
\end{proof}

\begin{corollary}\label{cor:tensor product of functions2}
Let $\W$, $\X_1$, $\X_2$, $\Y_1$, and $\Y_2$ be quantum sets. Let $F_1\: \W \to \X_1$, $F_2\: \W \to \X_2$, $G_1\: \X_1 \to \Y_1$ and $G_2\: \X_2 \to \Y_2$ be functions. Then 
\begin{itemize}
    \item[(a)] If $(F_1,F_2)$ exists, so does $(G_1\circ F_1,G_2\circ F_2)$;
    \item[(b)] If $(G_1\circ F_1,G_2\circ F_2)$ exists, and both $G_1$ and $G_2$ are injective, then $(F_1,F_2)$ exists.
\end{itemize}
In both cases, we have
$(G_1\circ F_1,G_2\circ F_2)=(G_1\times G_2)\circ (F_1,F_2).$
\end{corollary}
\begin{proof}
For (a), we have $ P^{\Y_1}\circ(G_1\times G_2)\circ (F_1,F_2)=G_1\circ P^{\X_1} \circ (F_1,F_2)=G_1\circ F_1$ by Proposition \ref{prop:tensor product of functions},
and similarly, $ P^{\Y_2}\circ(G_1\times G_2)\circ (F_1,F_2)=G_2\circ P^{\X_2} \circ (F_1,F_2)=G_2\circ F_2$. Again appealing to Proposition \ref{prop:tensor product of functions}, we conclude that $(G_1\circ F_1,G_2\circ F_2)$ exists and is equal to $(G_1\times G_2)\circ (F_1,F_2)$.

For (b), assume that $(G_1\circ F_1,G_2\circ F_2)$ exists. By \cite{Kornell2}*{Lemma 10.4}, this is equivalent to the statement that each element in the image of $(G_1\circ F_1)^\star=F_1^\star\circ G_1^\star$ commutes with each element in the image of $(G_2\circ F_2)^\star=F_2^\star\circ G_2^\star$. Since $G_1$ and $G_2$ are injective, $G_1^\star$ and $G_2^\star$ are surjective \cite{Kornell2}*{Propositions 8.1 and 8.4}; hence each element in the image of $F_1^\star$ and each element in the image of $F_2^\star$ commute with each other. Again appealing to \cite{Kornell2}*{Lemma 10.4}, we conclude that $(F_1,F_2)$ exists. By (a), it follows that $(G_1\circ F_1,G_2\circ F_2)=(G_1\times G_2)\circ (F_1,F_2).$
\end{proof}

We now verify an elementary inequality for operators subspaces, whose ultimate purpose is to facilitate computations involving binary relations between products of quantum sets. For this verification, we introduce the notations $\check x$ and $\hat x$ for vectors $x \in X$. Specifically, for each finite-dimensional Hilbert space $X$ and each vector $x \in X$, let $\check x \in L(\CC, X)$ be defined by $\lambda\mapsto \lambda x$, and let $\hat x\in L(X, \CC)$ be defined by $y\mapsto\langle x,y\rangle$. Furthermore, we write $\check X= \{\check x:x\in X\}= L(\CC,X)$ and  $\hat X= \{\hat x:x\in X\} = X^* = L(X,\CC)$.

\begin{lemma}\label{lem:subspaces of tensor product}
Let $X_1$, $X_2$, $Y_1$ and $Y_2$ be finite-dimensional Hilbert spaces, and let $V$ be a subspace of  $L(X_1\otimes X_2,Y_1\otimes Y_2)$. Let $V_1 = \rho_{Y_2}(\CC 1_{Y_1}\otimes\hat Y_2)V(\CC 1_{X_1}\otimes\check X_2)\rho_{X_2}^{-1}$, and let $V_2 = \lambda_{Y_1}(\hat Y_1\otimes \CC 1_{Y_2})V(\check X_1\otimes \CC 1_{X_2})\lambda_{X_1}^{-1}$, where $\rho$ and $\lambda$ are the right and left unitors in the category $\FdHilb$. Then, $V \leq V_1 \tensor V_2$.
\end{lemma}
\begin{proof}
Fix $v \in V$.
Choose orthonormal bases $\{x_{1i}\}_{i=1}^{n_1}$, $\{x_{2i}\}_{i=1}^{n_2}$, $\{y_{1i}\}_{i=1}^{m_1}$, $\{y_{2i}\}_{i=1}^{m_2}$ for $X_1$, $X_2$, $Y_1$, $Y_2$, respectively. Since $v\in L(X_1,Y_1)\otimes L(X_2,Y_2)$, we have $v=\sum_{\ell=1}^kb_\ell\otimes c_\ell$ for some $b_\ell\in L(X_1,Y_1)$ and $c_\ell\in L(X_2,Y_2)$. 
Since for any basis $\{e_i\}_{i=1}^n$ of any $n$-dimensional Hilbert space $H$, we have that $1_H=\sum_{i=1}^n\check e_i\hat e_i$, we find that
\begin{align*}
    v & = (1_{Y_1}\otimes 1_{Y_2})v(1_{X_1}\otimes 1_{X_2}) = \left(1_{Y_1}\otimes\sum_{j=1}^{m_2}\check y_{2j}\hat y_{2j}\right)v\left(1_{X_1}\otimes\sum_{i=1}^{n_2}\check x_{2i}\hat x_{2i}\right)
    \\
    & = \sum_{j=1}^{m_2}\sum_{i=1}^{n_2}\left(1_{Y_1}\otimes\check y_{2j}\hat y_{2j}\right)v\left(1_{X_1}\otimes\check x_{2i}\hat x_{2i}\right)  = \sum_{j=1}^{m_2}\sum_{i=1}^{n_2}(1_{Y_1}\otimes\check y_{2j}\hat y_{2j})\left(\sum_{\ell=1}^kb_\ell\otimes c_\ell\right)(1_{X_1}\otimes\check x_{2i}\hat x_{2i})
    \\
    & = \sum_{j=1}^{m_2}\sum_{i=1}^{n_2}\sum_{\ell=1}^kb_\ell\otimes\check y_{2j}\hat y_{2j}c_\ell\check x_{2i}\hat x_{2i}    
    = \sum_{j=1}^{m_2}\sum_{i=1}^{n_2}\sum_{\ell=1}^k[\rho_{Y_2}(b_\ell\otimes 1_\CC)\rho_{X_2}^{-1}]\otimes\check y_{2j}\hat y_{2j}c_\ell\check x_{2i}\hat x_{2i}   
 \\
& = \sum_{j=1}^{m_2}\sum_{i=1}^{n_2}\sum_{\ell=1}^k[\rho_{Y_2}(b_\ell\otimes \hat y_{2j}c_\ell\check x_{2i})\rho_{X_2}^{-1}]\otimes\check y_{2j}\hat x_{2i} 
\\
& = \sum_{j=1}^{m_2}\sum_{i=1}^{n_2}\sum_{\ell=1}^k[\rho_{Y_2}(1_{Y_1}\otimes\hat y_{2j})(b_\ell\otimes c_\ell)(1_{X_1}\otimes \check x_{2i})\rho_{X_2}^{-1}]\otimes\check y_{2j}\hat x_{2i} 
\\
& = \sum_{j=1}^{m_2}\sum_{i=1}^{n_2}[\rho_{Y_2}(1_{Y_1}\otimes\hat y_{2j})v(1_{X_1}\otimes \check x_{2i})\rho_{X_2}^{-1}]\otimes\check y_{2j}\hat x_{2i} 
\end{align*}
where the third-to-last equality follows because each operator $\hat y_{2j}c_\ell\check x_{2i}$ is a scalar.
We conclude that $v\in V_1 \otimes L(X_2,Y_2)$, and similarly, $v\in L(X_1,Y_1)\otimes V_2$. The intersection of $V_1 \otimes L(X_2,Y_2)$ and $L(X_1,Y_1)\otimes V_2$ is of course $V_1 \tensor V_2$, and thus $v \in V_1 \tensor V_2$. We vary $v \in V$ to conclude that $V \leq V_1 \tensor V_2$.
\end{proof}

\begin{lemma}\label{lem:relations on tensor products}
Let $\X_1$, $\X_2$, $\Y_1$ and $\Y_2$ be quantum sets, and let $R\in\qRel(\X_1\times\X_2,\Y_1\times\Y_2)$. Then,
$R\leq (P^{\Y_1}\circ R\circ P^{\X_1\dag})\times (P^{\Y_2}\circ R\circ P^{\X_2\dag}).$
\end{lemma}
\begin{proof}
Fix $X_1 \atomof \X_1$, $Y_1 \atomof \Y_1$, $X_2 \atomof \X_2$ and $Y_2 \atomof \Y_2$. Let $V = R(X_1 \tensor Y_1, X_2 \tensor Y_2)$, and define $V_1$ and $V_2$ as in Lemma \ref{lem:subspaces of tensor product}. We calculate that
\begin{align*}
(P^{\Y_1}\circ R\circ P^{\X_1\dag})(X_1, Y_1)
& \geq
P^{\Y_1}(Y_1 \tensor Y_2, Y_1) \cdot R(X_1 \tensor Y_1, X_2 \tensor Y_2) \cdot P(X_1 \tensor X_2, X_1)^\dag
\\ & =
\rho_{Y_2}(\CC 1_{Y_1}\otimes\hat Y_2)\cdot V \cdot (\CC 1_{X_1}\otimes\check X_2)\rho_{X_2}^{-1}
= V_1
\end{align*}
Similarly, $(P^{\Y_2} \circ R \circ P^{\X_2\dag})(X_2, Y_2) \geq V_2$. We now apply Lemma \ref{lem:subspaces of tensor product} to calculate that 
\begin{align*}
    R(X_1\otimes X_2,Y_1\otimes Y_2) & \leq V_1 \tensor V_2 \leq (P^{\Y_1}\circ R\circ P^{\X_1\dag})(X_1, Y_1) \tensor (P^{\Y_2} \circ R \circ P^{\X_2})(X_2, Y_2)
    \\ & =
    ((P^{\Y_1}\circ R\circ P^{\X_1\dag}) \times (P^{\Y_2}\circ R\circ P^{\X_2\dag}))(X_1\otimes X_2, Y_1 \otimes Y_2)
\end{align*}
We vary $X_1 \atomof \X_1$, $X_2 \atomof \X_2$, $Y_1 \atomof \Y_1$ and $Y_2 \atomof \Y_2$ to conclude that $R \leq (P^{\Y_1}\circ R\circ P^{\X_1^\dag}) \times (P^{\Y_2}\circ R\circ P^{\X_2\dag})$, as claimed.
\end{proof}

\begin{proposition}\label{prop:order between product maps}
Let $\W$, $\X_1$, $\X_2$, $\Y_1$ and $\Y_2$ be quantum sets. Let $F_1\: \W \to \X_1$, $F_2\: \W \to \X_2$, $G_1\: \W \to \Y_1$ and $G_2\: \W \to \Y_2$ be functions. If $(F_1, F_2)$ and $(G_1, G_2)$ both exist, then
$(G_1,G_2)\leq ((G_1\circ F_1^\dag)\times (G_2\circ F_2^\dag))\circ (F_1,F_2).$
\end{proposition}
\begin{proof}
By Lemma \ref{lem:relations on tensor products}, we have
\begin{align*}
(G_1,G_2)\circ (F_1,F_2)^\dag & \leq 
(P^{\Y_1}\circ (G_1,G_2)\circ (F_1,F_2)^\dag\circ P^{\X_1\dag})\times(P^{\Y_2}\circ (G_1,G_2)\circ (F_1,F_2)^\dag\circ P^{\X_2\dag})\\
& \quad \quad = (G_1 \circ F_1^\dag)\times(G_2\circ F_2^\dag);
\end{align*}
hence 
$(G_1,G_2)\leq (G_1,G_2)\circ (F_1,F_2)^\dag\circ (F_1,F_2)\leq \big((G_1 \circ F_1^\dag)\times(G_2\circ F_2^\dag)\big)\circ (F_1,G_1).$
\end{proof}

\section{The trace on binary relations}\label{appendix.A}
Every compact closed category has a canonically defined trace on each endomorphism set. For each quantum set $\X$, the trace $\Tr_\X\: \cat{qRel}(\X,\X) \to \cat{qRel}(\mathbf 1, \mathbf 1)$ is defined by $\Tr_\X(R) = E_\X \circ (R \times I_{\X^*}) \circ E_\X^\dagger$. The orthomodular lattice $\cat{qRel}(\mathbf 1, \mathbf 1)$ consists of two elements: $\bot \leq \top$. Note that $\Tr_{\mathbf 1}$ is just the identity on $\cat{qRel}(\mathbf 1, \mathbf 1).$

\begin{lemma}\label{tr}\label{appendix.A.1}
Let $\X$ be a quantum set. For each binary relation $R$ from $\X$ to $\X$, the equation $\Tr_\X(R) = \bot$ is equivalent to the equation $\Tr_X(R(X,X)) = 0$ for each Hilbert space $X \in \At(\X)$.
\end{lemma}

\begin{proof}
The equation $\Tr_\X(R) = \bot$ is equivalent to
$$\bigvee_{X_1, X_2 \in \At(\X)} \bigvee_{X_3, X_4 \in \At(\X)} E_\X(X_3 \tensor X_4^*, \CC) \cdot (R(X_1, X_3) \tensor I_{\X^*}(X_2^*,X_4^*)) \cdot E_\X^\dagger(\CC, X_1 \tensor X_2^*)=0,$$
by definition of composition and product for binary relations \cite{Kornell2}*{section 3}. Terms for which $X_3 \neq X_4$ do not contribute because $E_\X(X_3 \tensor X_4^*,\CC)=0$, terms for which $X_1 \neq X_2$ do not contribute because $E_\X^\dagger(\CC, X_1 \tensor X_2^*)=0$, and terms for which $X_2 \neq X_4$ do not contribute because $I_{\X^*}(X_2^*,X_4^*)=0$. Thus, $\Tr_\X(R) = \bot$ if and only if
$$\bigvee_{X \in \At(\X)} E_\X(X \tensor X^*, \CC) \cdot (R(X, X) \tensor I_{\X^*}(X^*,X^*)) \cdot E_\X^\dagger(\CC, X \tensor X)=0.$$
A sum of subspaces is equal to zero if and only if each subspace is equal to zero. Furthermore, the operator spaces $E_\X(X \tensor X^*, \CC)$, $E_\X^\dagger(\CC, X \tensor X)$, and $I_{\X^*}(X^*,X^*)$ are each spanned by a single operator. Thus, $\Tr_\X(R) = \bot$ if and only if $\counit_X \cdot (R(X,X) \tensor 1_{X^*}) \cdot \counit^\dagger_X=0$ for each atom $X \in \At(\X)$, where $\counit_X$ denotes the unit of the dagger compact category of finite-dimensional Hilbert spaces and linear operators. This completes the proof, because it is well known that $\counit_X \cdot (r \tensor 1_{X^*}) \cdot \counit^\dagger_X = \Tr_X(r)$ for each operator $r$ on a finite-dimensional Hilbert space $X$.
\end{proof}

\begin{proposition}\label{appendix.A.2}
Let $\X$ and $\Y$ be quantum sets, and let $R$ and $S$ be binary relations from $\X$ to $\Y$. Then, $R \perp S$ if and only if $\Tr_\X(S^\dagger \circ R) = \bot$.
\end{proposition}

\begin{proof}
We follow a chain of equivalences, with Lemma \ref{tr} used for the first equivalence.
\begin{align*}
\Tr_\X(S^\dagger \circ R) = \bot
& \;\Leftrightarrow\;
\forall X \in \At(\X).\,\Tr_X((S^\dagger\circ R)(X,X)) = 0
\\ & \;\Leftrightarrow\;
\forall X \in \At(\X).\,\Tr_X\left(\sum_{Y \in \At(\Y)} S(X,Y)^\dagger \cdot R(X,Y)\right) = 0
\\ & \;\Leftrightarrow\;
\forall X \in \At(\X).\,\forall Y \in \At(\Y).\,\Tr_X( S(X,Y)^\dagger \cdot R(X,Y)) = 0
\\ & \;\Leftrightarrow\;
\forall X \in \At(\X).\,\forall Y \in \At(\Y).\,R(X,Y) \perp S(X,Y) = 0
\\ & \;\Leftrightarrow\;
R \perp S \qedhere
\end{align*}
\end{proof}

For $P, Q \in \cat{qRel}(\Y,\mathbf 1)$, we set $\X = \mathbf 1$, $R = P^\dagger$, and $S = Q^\dagger$, to find that $P \perp Q$ if and only if $Q \circ P^\dagger = \bot$.

\end{document}